\documentclass[12pt,a4paper]{article}
\usepackage[numbers]{natbib}
\usepackage{amsmath,amssymb,amsfonts,amsthm}
\usepackage{graphicx}
\usepackage{dsfont}
\usepackage{float}
\usepackage{url}
\usepackage{multirow}
\theoremstyle{plain}
\newtheorem{theorem}{Theorem}[section]
\newtheorem{proposition}{Proposition}[section]
\newtheorem*{lemma}{Lemma}

\newtheorem{remark}{Remark}[section]

\newtheorem{example}{Example}[section]
\newtheorem{assumption}{Assumption}[section]
\numberwithin{equation}{section}
\numberwithin{figure}{section}
\numberwithin{table}{section}
\newcommand{\1}{\mathds{1}}
\newcommand{\co}{\mathrm{co}}
\newcommand{\Cov}{\mathrm{Cov}}
\newcommand{\diag}{\mathrm{diag}}
\newcommand{\Int}{\mathrm{Int}}
\newcommand{\spn}{\mathrm{span}}
\newcommand{\tr}{\mathrm{tr}}
\newcommand{\Unif}{\mathcal{U}}

\newcommand{\Vol}{\mathrm{Vol}}
\newcommand{\R}{\mathbb{R}}
\renewcommand{\S}{\mathbb{S}}
\newcommand{\B}{{\overline B}}
\newcommand{\G}{{\overline G}}
\newcommand{\N}{{\mathcal{N}}}
\newcommand{\E}{{\mathrm{E}}}
\renewcommand{\P}{{\overline P}}
\newcommand{\PP}{{\overline P_{\mathrm{tube}}}}
\newcommand{\dd}{{\mathrm{d}}}
\begin{document}
\title{Simultaneous confidence bands for contrasts between several nonlinear regression curves}
\author{Xiaolei Lu\footnote{Department of Statistical Sciences, Graduate University for Advanced Studies, 10-3 Midoricho, Tachikawa, Tokyo 190-8562, Japan, Email: {\tt lu.xiaolei@ism.ac.jp}}, \ \ %
Satoshi Kuriki\footnote{The Institute of Statistical Mathematics, 10-3 Midoricho, Tachikawa, Tokyo 190-8562, Japan, Email: {\tt kuriki@ism.ac.jp}}
}

\date{}

\maketitle

\begin{abstract}
We propose simultaneous confidence bands of the hyperbolic-type for the contrasts between several nonlinear (curvilinear) regression curves.
The critical value of a confidence band is determined from the distribution of the maximum of a chi-square random process defined on the domain of explanatory variables.
We use the volume-of-tube method to derive an upper tail probability formula of the maximum of a chi-square random process, which is asymptotically exact and sufficiently accurate in commonly used tail regions. 
Moreover, we prove that the formula obtained is equivalent to the expectation of the Euler-Poincar\'e characteristic of the excursion set of the chi-square random process, and hence conservative.
This result is therefore a generalization of Naiman's inequality for Gaussian random processes.
As an illustrative example, growth curves of consomic mice are analyzed.

\end{abstract}

\section{Introduction}
\label{sec:introduction}

This paper concerns multiple comparisons of $k$ ($\ge 3$) nonlinear (curvilinear) regression curves estimated from independent $k$ groups.
Suppose that for each group $i=1,\ldots,k$, and for each explanatory variable $x_j\in\mathcal{X}$, $j=1,\ldots,n$, we have observations $y_{ij1},\ldots,y_{ij r_i}$ as objective variables with $r_i$ replications,
which are assumed to follow the model
\begin{align}
\label{model}
y_{ijh} = g_i(x_j) + \varepsilon_{ijh}, \quad i=1,\ldots,k,\ \ j=1,\ldots,n,\ \ h=1,\ldots,r_i.
\end{align}
Here, $\mathcal{X}\subseteq\R$ is the domain of explanatory variables, and
random errors $\varepsilon_{ijh}$ are assumed to be independently distributed as the normal distribution $\N(0,\sigma(x_j)^2)$.
The variance function $\sigma(x)^2$ is supposed to be known, or at least known up to a constant $\sigma(x)^2 =\sigma^2\sigma_0(x)^2$.
In the case of the latter, we suppose that an independent estimator $\widehat\sigma^2$ of $\sigma^2$ is available.  
In addition, we assume that the true regression curve has the form
\begin{equation}
\label{beta_f}
 g_i(x)=\beta_i^\top f(x),\ \ x\in\mathcal{X},
\end{equation}
where $f(x)=(f_1(x),\ldots,f_p(x))^\top$ is a known regression basis vector function, $\beta_i=(\beta_{i1},\ldots,\beta_{ip})^\top$ is an unknown parameter vector.
Then, the least squares estimator $\widehat\beta_i$ of $\beta_i$ has the multivariate normal distribution $\N_p(\beta_i,r_i^{-1}\Sigma)$,
where
\[
 \Sigma = \left(\sum_{j=1}^n \frac{1}{\sigma(x_j)^2}f(x_j)f(x_j)^\top\right)^{-1}
\]
is the inverse of the $p\times p$ information matrix.
When $\sigma(x)^2=\sigma^2\sigma_0(x)^2$, we have $\Sigma=\sigma^2\Sigma_0$, where $\Sigma_0$ is $\Sigma$ with $\sigma(x_j)$ replaced by $\sigma_0(x_j)$.

Let $\mathcal{C}$ denote the set of vectors $c=(c_1,\ldots,c_k)^\top$ such that $\sum_{i=1}^k c_i=0$.
The focus of this paper is the construction of $1-\alpha$ simultaneous confidence bands for all the contrasts $\sum_{i=1}^k c_i g_i(x) = \sum_{i=1}^k c_i \beta_i^\top f(x)$ between the $k$ regression curves for all $x \in \mathcal{X}$ and $c\in\mathcal{C}$,
where $\mathcal{X}\subset\R$ is a finite or half-infinite interval, a finite union of such intervals, or an infinite interval $(-\infty,\infty)$.
Specifically, according to the traditional form of the point estimate plus or minus a probability point times the estimated standard error,
we construct a $1-\alpha$ simultaneous confidence band of the form
\begin{align}
\label{bands}
 \sum_{i=1}^k c_i \beta_i^\top f(x) \in \sum_{i=1}^k c_i\widehat\beta_i^\top f(x) \pm b_{1-\alpha}
 \sqrt{\left(\sum_{i=1}^k \frac{c_i^2}{r_i}\right) f(x)^\top\Sigma f(x)},
\end{align}
where $\widehat\beta_i^\top f(x)$ is the estimator of $\beta_i^\top f(x)$ in (\ref{beta_f}).
This form is referred to as a hyperbolic-type (\citep{Liu10}).
The critical value $b_{1-\alpha}$ is determined such that the event
 in (\ref{bands}) for all $x \in \mathcal{X}$ and $c \in \mathcal{C}$ holds with a probability of at least $1-\alpha$.
Our problem typically arises from growth curve analysis and longitudinal data analysis.

Throughout this paper, we assume that the regression curve $g_i(x)$ is a linear combination of a finite number of known basis functions in (\ref{beta_f}).
Although it is a conventional regression model, we must always be careful regarding the approximation bias caused by model misspecification.
This issue is examined in Section \ref{sec:simulation}.

The problem concerning the construction of simultaneous confidence bands in a regression model originates with Working and Hotelling \cite{Working-Hotelling29}.
They formalized this problem as the construction of confidence intervals for an estimated regression line, and provided a critical value by making use of the Cauchy-Schwarz inequality.
Specifically, 
Working and Hotelling \cite{Working-Hotelling29} treated the case of
\begin{itemize}
\itemsep=-1mm
\item[(i)]
one regression model (equivalent to the case $k=2$ in our problem),

\item[(ii)]
the simple regression $f(x)=(1,x)^\top$, and

\item[(iii)]
the unrestricted domain of the explanatory variables
$\mathcal{X}=(-\infty,\infty)$.
\end{itemize}
Subsequently, many reports concerning the relaxation of these conditions have appeared in literature.

In the case of one regression model, Wynn and Bloomfield \cite{Wynn-Bloomfield71} pointed out that the use of the Cauchy-Schwarz inequality leads to conservative bands unless both (ii) and (iii) hold.
They illustrated improved confidence bands for the quadratic regression $f(x)=(1,x,x^2)^\top$.
Uusipaikka \cite{Uusipaikka83} constructed exact confidence bands for linear regression when $\mathcal{X}$ is a finite interval.
See Liu et al.\ \cite{Liu-Lin-Piegorsch08} and Liu \cite{Liu10} for historical reviews.
The problem of $k$ ($\ge 3$) regression curve comparisons was considered by Spurrier \cite{Spurrier99, Spurrier02} and Lu and Chen \cite{Lu-Chen09}, who proposed procedures based on simple linear regression.
However, it is difficult to extend these methods to nonlinear regression.

One exception is Naiman's \cite{Naiman86} integral-geometric approach.
In the unit sphere $\S^{p-1}$ of the $p$-dimensional Euclidean space,
he defined a trajectory
\begin{equation}
\label{Gamma}
 \Gamma = \overline{\{ \psi(x) \mid x\in \mathcal{X} \}} \subset \S^{p-1}
\end{equation}
of a normalized basis vector function
\begin{equation}
\label{psi}
 \psi(x) = \frac{\Sigma^{1/2} f(x)}{\Vert\Sigma^{1/2} f(x)\Vert},
\end{equation}
and evaluated the volume of the tubular neighborhood of $\Gamma$.
In the case of one regression model, he constructed a simultaneous confidence band with the critical value obtained from this volume.
The volume formula for such tubes originated from Hotelling \cite{Hotelling39} and Weyl \cite{Weyl39}.
Currently, this idea is understood in the volume-of-tube method framework
(\cite{Adler-Taylor07}, \cite{Kuriki-Takemura01}, \cite{Kuriki-Takemura09}, \cite{Sun93}, \cite{Takemura-Kuriki02}).
As shown in Section \ref{sec:pivot},
we require the tail probability of the maximum of a Gaussian random field or chi-square random process as a pivotal quantity.
The volume-of-tube method is a methodology to evaluate such tail probabilities.

In this paper, we adopt this integral-geometric approach.
In the case of $k\ge 3$,
we define a subset $M$ in (\ref{M}) of a unit sphere, and by evaluating the volume of its tubular neighborhood, we obtain the critical value $b_{1-\alpha}$ in (\ref{bands}) by means of the volume-of-tube method.
Moreover, we prove that the proposed confidence band is conservative.
It is known that Naiman's \cite{Naiman86} confidence band is conservative (Naiman's inequality, see also \cite{Johnstone-Siegmund89}), and our result is regarded as its generalization.

Note that, in the setting of this paper, the covariance matrices of the estimators $\widehat\beta_i$ are identical up to a multiplicative constant.
This property arises from the condition that the explanatory variables $x_j$ are common between $k$ groups in the model (\ref{model}).
This represents the so-called the balanced case.
For the unbalanced case, the problem of constructing simultaneous confidence bands is quite tedious and only simulation-based approaches are available
(\cite{Jamshidian-Liu-Bretz10}, \cite{Liu10}, \cite{Liu-Jamshidian-Zhang04}, \cite{Liu-Wynn-Hayter08}).
In this paper, we address only the balanced case.

Moreover, note that in the one-group case ($k=1$), various simultaneous confidence bands by means of the volume-of-tube method have been proposed.
Johansen and Johnstone \cite{Johansen-Johnstone90} demonstrated the usefulness of Hotelling's volume formula for the construction of simultaneous bands.
The application to the B-spline regression is found in Zhou et al.\ \cite{Shen-Wolfe-Zhou98}.
Sun and Loader \cite{Sun-Loader94} proposed a modification to the volume-of-tube formula when a small approximation bias caused by model misspecification exists.
In succeeding papers, Sun and her coauthors developed this idea further in various model settings 
 (\cite{Faraway-Sun95}, \cite{Sun-Loader-McCormick00}, \cite{Sun-Raz-Faraway99}).
See also Krivobokova et al.\ \cite{Krivobokova-Kneib-Claeskenset10}.
The crucial difference between this paper and existing work is that in this paper, we need to treat a Gaussian random field with a general dimensional ($k-1$ dimensional) index set, and need the volume formula up to an arbitrary order.

The layout of this paper is as follows.
In Section \ref{sec:pivot}, we define a Gaussian random field and a chi-square random process as pivotal quantities.
We show that the critical value $b_{1-\alpha}$ is determined from the upper tail probability of the maximum of a Gaussian random field or chi-square random process.
In Section \ref{sec:tube}, the volume-of-tube method and its related method known as the expected Euler-characteristic heuristics are briefly summarized.
The primary results are provided in Section \ref{sec:main}.
Some simulation study under model misspecification is conducted in Section \ref{sec:simulation}.
Section \ref{sec:growth} is devoted to the growth curve data analysis.
Proof details are located in the Appendix.

\section{Random fields as pivotal quantities}
\label{sec:pivot}

Our problem is to determine the critical value $b_{1-\alpha}$ in (\ref{bands}).
First, assume that $\Sigma$ is fully known.
Define a pivotal quantity
\begin{equation}
\label{Txc}
T(x,c)
= \frac{\sum_{i=1}^k c_i(\widehat\beta_i-\beta_i)^\top f(x)}{\sqrt{\left(\sum_{i=1}^k\frac{c_i^2}{r_i}\right) f(x)^\top\Sigma f(x)}}.
\end{equation}
Then, the critical value $b_{1-\alpha}$ is the solution $b$ of the equation
\begin{align*}
\Pr\bigl(T(x,c)\le b,\,\forall x \in \mathcal{X},\,\forall c \in \mathcal{C}\bigr)
= \Pr\biggl(\max_{x\in\mathcal{X},c\in\mathcal{C}} T(x,c)\le b\biggr)
= 1-\alpha.
\end{align*}
In this expression, we use $T(x,c)$ instead of $|T(x,c)|$, because $c\in\mathcal{C}$ implies $-c\in\mathcal{C}$ and $|T(x,c)|$ is equal to $T(x,c)$ or $T(x,-c)$.
Inverting $|T(c,x)|\le b_{1-\alpha}$ yields the $1-\alpha$ simultaneous confidence band in (\ref{bands}).

In the following, we show that $b_{1-\alpha}^2$ is the upper $\alpha$ point of the maximum of a chi-square random process.
We can assume that $\sum_{i=1}^k c_i^2/r_i=1$ without loss of generality, 
because $T(x,c)$ is a homogeneous function in $c$.
Let $\rho=(\sqrt{r_1},\ldots,\sqrt{r_k})^\top$,
and define a $k\times (k-1)$ matrix $H$ such that
$\rho^\top H=0$, $H^\top H=I_{k-1}$, and $H H^\top = I_k - \rho \rho^\top/(\rho^\top \rho)$.
(An example of $H$ is given in Remark \ref{rem:H} below.)
Then the $c=(c_1,\ldots,c_k)^\top$ such that $\sum_{i=1}^k c_i^2/r_i=1$ and $\sum_{i=1}^k c_i=0$ are represented as
\[
 c = \diag(\sqrt{r_1},\ldots,\sqrt{r_k}) H h,\ \ h\in\S^{k-2},
\]
where $\S^{k-2}$ is the set of ($k-1$)-dimensional unit column vectors.

Let $\Sigma^{1/2}$ be a matrix such that
 $(\Sigma^{1/2})^\top \Sigma^{1/2} = \Sigma$, and let $\Sigma^{-1/2}$ be its inverse.
Then, $\eta_i=\sqrt{r_i}(\Sigma^{-1/2})^\top(\widehat\beta_i-\beta_i)$ is distributed normally as $\N_p(0,I)$, independently for $i=1,\ldots,k$.
Let $\psi:\mathcal{X}\to\S^{p-1}$ as defined in (\ref{psi}).
Then, $T(x,c)$ is rewritten as 
\begin{align}
T(x,c)
=& \sum_{i=1}^k \frac{c_i}{\sqrt{r_i}} \sqrt{r_i} \{(\Sigma^{-1/2})^\top(\widehat\beta_i-\beta_i)\}^\top \frac{\Sigma^{1/2} f(x)}{\Vert\Sigma^{1/2} f(x)\Vert} \nonumber \\
=& c^\top \diag(\sqrt{r_1},\ldots,\sqrt{r_k})^{-1}
 \begin{pmatrix} \eta_1^\top \\ \vdots \\ \eta_k^\top \end{pmatrix}_{k\times p} \psi(x) \nonumber \\
=& h^\top \begin{pmatrix} \xi_1^\top \\ \vdots \\ \xi_{k-1}^\top \end{pmatrix}_{(k-1)\times p} \psi(x) \nonumber \\
=& \xi^\top \{h\otimes\psi(x)\},
\label{Txc2}
\end{align}
where
$\xi_i$ are $p\times 1$ vectors defined by  
$(\xi_1,\ldots,\xi_{k-1})_{p\times(k-1)}=(\eta_1,\ldots,\eta_k)_{p\times k}H$,
$\xi = (\xi_1^\top,\ldots,\xi_{k-1}^\top)^\top$ is a $p(k-1)\times 1$ vector,
and `$\otimes$' is the Kronecker product.
Vectors $\eta_i$ consist of independent standard Gaussian random variables $\N(0,1)$, therefore, so does vector $\xi$.
When $x$ and $h$ are fixed, because $\Vert\psi(x)\Vert=\Vert h\otimes\psi(x)\Vert=1$, $\xi_i^\top \psi(x)$ is distributed as $\N(0,1)$ independently for $i=1,\ldots,k$,
and $\xi^\top \{h\otimes\psi(x)\}$ is distributed as $\N(0,1)$.

From (\ref{Txc2}), we can see that
\begin{equation}
\label{maxT}
 \max_{c\in\mathcal{C}} T(x,c) = \sqrt{\sum_{i=1}^{k-1} \bigl\{\xi_i^\top \psi(x)\bigr\}^2}.
\end{equation}
For each fixed $x$, this is distributed as the square root of the chi-square distribution $\chi^2_{k-1}$ with $k-1$ degrees of freedom.

When $\Sigma=\sigma^2\Sigma_0$ with $\Sigma_0$ known,
and an independent estimator $\widehat\sigma^2\sim\sigma^2\chi^2_\nu/\nu$ of unknown $\sigma^2$ is available,
we redefine $T(x,c)$ in (\ref{Txc}) by replacing $\Sigma$ in the denominator with $\widehat\sigma^2\Sigma_0$.
Thus, instead of (\ref{Txc2}) and (\ref{maxT}), we have
\[
 T(x,c) = \frac{1}{\tau}\xi^\top \{h\otimes\psi(x)\}, \quad 
 \max_{c\in\mathcal{C}} T(x,c) = \sqrt{\frac{1}{\tau^2} \sum_{i=1}^{k-1} \bigl\{\xi_i^\top \psi(x)\bigr\}^2}, \quad \tau^2=\frac{\widehat\sigma^2}{\sigma^2}.
\]

Now, we consider the object in (\ref{Txc2}) as a random function of $(x,h)$:
\begin{equation}
\label{Z}
 Z(x,h) = \xi^\top \{h\otimes\psi(x)\},\ \ (x,h)\in\mathcal{X}\times\S^{k-2},
\end{equation}
where $\xi\sim \N_{p(k-1)}(0,I)$.
Then, $Z(x,h)$ is the Gaussian random field with mean 0, variance 1, and covariance function
\[
 \Cov\bigl(Z(x,h),Z(\tilde x,\tilde h)\bigr) = \psi(x)^\top\psi(\tilde x) \cdot h^\top \tilde h.
\]
Similarly, we define the chi-square random process with $k-1$ degrees of freedom:
\begin{equation}
\label{Y}
 Y(x) = \sum_{i=1}^{k-1} \bigl\{\xi_i^\top\psi(x)\bigr\}^2,\ \ x\in\mathcal{X}.
\end{equation}

We summarize the results of this section below.
\begin{theorem}
When $\Sigma$ is known, the critical value $b_{1-\alpha}$ is determined as the solution $b=b_{1-\alpha}$ of
\[
 \Pr\biggl(\max_{x\in\mathcal{X},h\in\S^{k-2}} Z(x,h)\ge b \biggr)
=  \Pr\biggl(\max_{x\in\mathcal{X}} Y(x)\ge b^2\biggr) = \alpha,
\]
where $Z(x,h)$ is the Gaussian random field defined in (\ref{Z}),
and $Y(x)$ is the chi-square random process defined in (\ref{Y}).

When $\Sigma=\sigma^2\Sigma_0$ with $\Sigma_0$ known, the critical value $b_{1-\alpha}$ is determined as the solution $b=b_{1-\alpha}$ of
\[
 \E\biggl[\Pr\biggl(\max_{x\in\mathcal{X},h\in\S^{k-2}} Z(x,h)\ge b\tau \,\big|\, \tau^2 \biggr)\biggr]
=  \E\biggl[\Pr\biggl(\max_{x\in\mathcal{X}} Y(x)\ge b^2\tau^2 \,\big|\, \tau^2 \biggr)\biggr] = \alpha,
\]
where the expectation is taken over $\tau^2\sim\chi^2_\nu/\nu$, with $\nu$ being the degrees of freedom of the estimator of $\sigma^2$.
\end{theorem} 

\begin{remark}
\label{rem:H}
An example of $k\times (k-1)$ matrix $H$ such that $\rho^\top H=0$,
$H^\top H=I_{k-1}$, $H H^\top = I_k - \rho \rho^\top/(\rho^\top \rho)$ with $\rho=(\sqrt{r_1},\ldots,\sqrt{r_k})^\top$ is given as
\[
 H = \begin{pmatrix}
 \frac{\sqrt{r_1 r_2}}{\sqrt{R_1 R_2}} & \frac{\sqrt{r_1 r_3}}{\sqrt{R_2 R_3}} & \dots & \frac{\sqrt{r_1 r_k}}{\sqrt {R_{k-1}R_k}} \\
 -\frac{R_1}{\sqrt{R_1 R_2}} & \frac{\sqrt{r_2 r_3}}{\sqrt{R_2 R_3}} & \dots & \frac{\sqrt{r_2 r_k}}{\sqrt{R_{k-1}R_k}} \\
 & -\frac{R_2}{\sqrt{R_2 R_3}} & \dots & \frac{\sqrt{r_3 r_k}}{\sqrt{R_{k-1}R_k}} \\
 & & \ddots & \vdots \\
 0 & & & -\frac{R_{k-1}}{\sqrt{R_{k-1}R_k}}
 \end{pmatrix}_{k\times (k-1)},
\]
where $R_i=\sum_{j=1}^i r_j$.
\end{remark}

\section{Preliminaries on the volume-of-tube method}
\label{sec:tube}

In this section, we summarize the volume-of-tube method for evaluating the upper tail probability of the maximum of a Gaussian random field.

Let $\xi$ be a Gaussian random vector distributed as $\N_n(0,I)$.
Let $M$ be a closed subset of $\S^{n-1}$, which is the unit sphere (the set of unit column vectors) of $\R^n$.
Then, the random map $u\mapsto \xi^\top u$, $u\in M$, is a Gaussian random field with mean 0, variance 1, and covariance function $\Cov(\xi^\top u,\xi^\top v)=u^\top v$.
The volume-of-tube method approximates the distribution of the maximum $\max_{u\in M} \xi^\top u$.
The maximum of $Z(x,h)$ in (\ref{Z}) can be treated in this framework by setting
\begin{equation}
\label{M}
M = \overline{\{ h\otimes\psi(x) \mid (x,h)\in\mathcal{X}\times\S^{k-2} \}}\quad\mbox{and}\quad n=p(k-1).
\end{equation}
The dimension of $M$ is $d=\dim M=k-1$.
To apply the volume-of-tube method, we require the following assumption on $M$.
Let the symbol `$\sqcup$' denote disjoint union.

\begin{assumption}
\label{as:M}
$M$ is a $d$-dimensional closed piecewise $C^2$-manifold, or 
$M$ is a $d$-dimensional $C^2$-manifold with piecewise $C^2$-boundary.
We write $M=\Int M \sqcup \partial M$, where $\Int M$ and $\partial M$ denote the interior and boundary of $M$, respectively.
In the former case, $\partial M=\emptyset$.
\end{assumption}

When $M$ is defined by (\ref{M}), we can provide a sufficient condition for Assumption \ref{as:M}.
\begin{assumption}
\label{as:psi}
$\psi:\mathcal{X}\to\S^{p-1}$ is a one-to-one map of class piecewise $C^2$.
There does not exist $x,\tilde x\in\mathcal{X}$ such that $\psi(x)=-\psi(\tilde x)$.
\end{assumption}
Under Assumption \ref{as:psi}, the map $(x,h)\mapsto h\otimes\psi(x)$ is a piecewise $C^2$ one-to-one map.

\begin{example}
Consider the polynomial regression with a basis function vector $f(x)=(1,x,\ldots,x^{p-1})^\top$.
When the domain of $x$ is a finite interval $\mathcal{X}=[a,b]$,
we have
\begin{align*}
 \Int M =& \{ h\otimes \psi(x) \mid x\in (a,b),\,h\in\S^{k-2} \}, \\
 \partial M =& \{ h\otimes \psi(a) \mid h\in\S^{k-2} \} \sqcup \{ h\otimes \psi(b) \mid h\in\S^{k-2} \}.
\end{align*}
When $\mathcal{X}=(-\infty,\infty)$,
$\psi(\pm\infty)=(\pm 1)^{p-1}\Sigma^{1/2} e_p/\sqrt{e_p^\top\Sigma e_p}$ with $e_p=(0,\ldots,0,1)^\top$, and hence $h\otimes \psi(\infty)=(-1)^{p-1}h\otimes\psi(-\infty)$.
This denotes that $M$ is a closed manifold without boundary.
\end{example}

\begin{example}
Consider the trigonometric regression with a basis function vector
\[
 f(x)=\left(1,\sqrt{2}\cos x,\sqrt{2}\sin x,\ldots,\sqrt{2}\cos mx,\sqrt{2}\sin mx\right)^\top.
\]
When $\mathcal{X}=[0,2\pi)$, $M$ is a closed manifold without boundary.
\end{example}

We now define ``tube'', the key concept of the volume-of-tube method.
The set of $\S^{n-1}$ points whose great circle distance from $M$ is less than or equal to $\theta$ is the tube about $M$ with radius $\theta$, and has the expression
\[
 M_\theta = \Bigl\{ v\in \S^{n-1} \,\big|\, \min_{u\in M} \cos^{-1} \bigl(u^\top v\bigr)\le\theta \Bigr\}.
\]

If the radius $\theta$ is sufficiently small, the tube $M_\theta$ does not have self-overlap.
Whereas, when the radius $\theta$ is large, the tube does have self-overlap.
The threshold radius between the two cases is known as the critical radius $\theta_{\mathrm{c}}$.
We let $\theta_{\mathrm{c}}=\pi/2$ when the threshold radius is more than $\pi/2$.
Under Assumption \ref{as:M}, we can prove that $\theta_{\mathrm{c}}>0$.
Figure 1 in Kuriki and Takemura \cite{Kuriki-Takemura09} depicts an example of a tube and its critical radius.

The support cone (or tangent cone) of $M$ at $u\in M$ is denoted by $S_u M$.
(See Section 1.2 of Takemura and Kuriki \cite{Takemura-Kuriki02} for the definition.)
The cone with base set $M$ is denoted by $\co(M) = \bigsqcup_{\lambda\ge 0} \lambda M$.
Then, the support cone of $\co(M)$ at $u\in M$ is decomposed as
$S_u(\co(M)) = S_u M\oplus\spn\{u\}$, where $\spn\{u\}$ is the linear space spanned by $u$.
The normal cone of $\co(M)$ at $u\in M$ is defined by the dual of the support cone:
$N_u(\co(M)) = S_u(\co(M))^*$.

Note that the ($m-1$)-dimensional volume of $\S^{m-1}$ is
$\Omega_m=2\pi^{m/2}/\Gamma(m/2)$.
For $m\times m$ matrix $A=(a_{ij})$, let $\tr_0 A=1$ and
\[
 \tr_e A = \sum_{1\le k_1<\ldots<k_e\le m} \det(a_{k_i k_j})_{1\le i,j\le e},\ \ 1\le e\le m
\]
(\cite{Muirhead05}, Appendix A.7).
Note that $\tr_1 A=\tr A$, $\tr_m A=\det A$.
The upper probability of the chi-square distribution with $m$ degrees of freedom is denoted by $\G_m(\cdot)$.
Now we can provide the upper tail probability formula for the Gaussian field $\xi^\top u$, $u\in M$.
The theorem below is a special case of Proposition 2.2 of Takemura and Kuriki \cite{Takemura-Kuriki02}.

\begin{proposition}
As $b\to\infty$,
\begin{equation}
\label{p}
\Pr\biggl(\max_{u\in M} \xi^\top u\ge b\biggr) =
 \PP(b) + O(\G_n(b^{2}(1+\tan^2\theta_{\mathrm{c}}))),
\end{equation}
where
\begin{equation}
\label{p-hat}
\PP(b)
= \sum_{0\le e\le d,\,e:\rm even } w_{d+1-e} \G_{d+1-e}(b^2)
+ \sum_{0\le e\le d-1} w'_{d-e} \G_{d-e}(b^2),
\end{equation}
with
\begin{align}
 w_{d+1-e} =& \frac{1}{\Omega_{d+1-e}\Omega_{n-d-1+e}}\int_{\Int M}\biggl\{\int_{N_u(\co(M))\cap\S^{n-1}} \tr_e H(u,v) \,\dd v\biggr\} \,\dd u,
\label{w} \\
 w'_{d-e} =& \frac{1}{\Omega_{d-e}\Omega_{n-d+e}}\int_{\partial M}\biggl\{\int_{N_u(\co(M))\cap\S^{n-1}} \tr_e H'(u,v) \,\dd v\biggr\} \,\dd u.
\label{w'}
\end{align}
Here, $H(u,v)$ is the second fundamental form of $\Int M$ at $u$ in the direction of $v$, and
$H'(u,v)$ is the second fundamental form of $\partial M$ at $u$ in the direction of $v$.
$\dd u$ is the volume element of $\Int M$ or $\partial M$, and $\dd v$ is the volume element of $N_u(\co(M))\cap\S^{n-1}$.
\end{proposition} 

In (\ref{p}), because $\theta_{\mathrm{c}}>0$, the error term $O\bigl(\G_n(b^{2}(1+\tan^2\theta_{\mathrm{c}}))\bigr)=O(b^{n-2}e^{-b^2(1+\tan^2\theta_{\mathrm{c}})/2})$ is exponentially smaller than each term $\G_j(b^2)=O(b^{j-2}e^{-b^2/2})$.
Hence, (\ref{p-hat}) can be used as an approximation formula when $b$ is large.
The method in which $\PP(b)$ is used as an approximate value is referred to as the volume-of-tube method, or simply the tube method.
This name comes from the volume formula for $M_\theta$ below.
\begin{remark}
\label{prop:tube}
For the radius $\theta\in [0,\theta_{\mathrm{c}}]$, the ($n-1$)-dimensional spherical volume of the tube $M_\theta$ is given by
\begin{align*}
 \Vol_{n-1}(M_\theta)
= \Omega_n \Biggl\{ & \sum_{0\le e\le d,\,e:\rm even } w_{d+1-e} \B_{\frac{1}{2}(d+1-e),\frac{1}{2}(n-d-1+e)}(\cos^2\theta) \\
& +\sum_{0\le e\le d-1} w'_{d-e} \B_{\frac{1}{2}(d-e),\frac{1}{2}(n-d+e)}(\cos^2\theta) \Biggr\},
\end{align*}
where $w_{d+1-e}$ and $w'_{d-e}$ are given in (\ref{w}) and (\ref{w'}),
$\B_{a,b}(\cdot)$ is the upper probability of the beta distribution with parameter $(a,b)$.
\end{remark}

The critical radius $\theta_{\mathrm{c}}$ can be evaluated using the following characterization (Theorem 4.18 of Federer \cite{Federer59}, Proposition 4.3 of Johansen and Johnstone \cite{Johansen-Johnstone90}, Lemma 2.2 of Takemura and Kuriki \cite{Takemura-Kuriki02}).
For a proof, see Theorem 2.9 of Kuriki and Takemura \cite{Kuriki-Takemura09}.
\begin{proposition}
The critical radius $\theta_{\mathrm{c}}$ of $M$ is given by
\begin{equation}
\label{critical}
 \tan^2\theta_{\mathrm{c}} = \inf_{u,v\in M,\,u\ne v} \frac{(1-u^\top v)^2}{\Vert P^\perp_v(u-v)\Vert^2},
\end{equation}
where $P^\perp_v$ is the orthogonal projection onto the normal cone $N_v(\co(M))$ of $\co(M)$ at $v$.
\end{proposition}

The local critical radius $\theta_{\mathrm{c,loc}}$ is defined as 
\begin{equation}
\label{critical-local}
 \tan^2\theta_{\mathrm{c,loc}} = \liminf_{u,v\in M,\,u\ne v,\,\Vert u-v\Vert\to 0} \frac{(1-u^\top v)^2}{\Vert P^\perp_v(u-v)\Vert^2}.
\end{equation}
From the definition, it holds that $\theta_{\mathrm{c}}\le\theta_{\mathrm{c,loc}}$.
In general, $\theta_{\mathrm{c,loc}}$ is easier to evaluate than $\theta_{\mathrm{c}}$.

We have summarized the volume-of-tube method to evaluate the upper tail probabilities of the maximum of random fields thus far.
There is another method utilized for the same purpose, known as the expected Euler-characteristic heuristic
 (\cite{Adler-Taylor07}, \cite{Worsley95}).
When applied to the Gaussian random field $\xi^\top u$, $u\in M$, this method is stated as follows.
For each $b$, define the excursion set by
\[
 A_b = \{ u\in M \mid \xi^\top u \ge b \}.
\]
Let $\chi(\cdot)$ be the Euler-Poincar\'e characteristic of a set,
and $\1(\cdot)$ be the indicator function for an event.
The expected Euler-characteristic heuristic assumes that
$\1(A_b\ne\emptyset) \approx \chi(A_b)$ for large $b$,
and
\[
 \Pr\biggl( \max_{u\in M} \xi^\top u \ge b \biggr) = \E[\1(A_b\ne\emptyset)] \approx \E[\chi(A_b)].
\]
Note that $\chi(A_b)$ can be evaluated by Morse's theorem, and is more tractable than $\1(A_b\ne\emptyset)$.
Takemura and Kuriki \cite{Takemura-Kuriki02} proved the equivalence of the volume-of-tube method and the expected Euler-characteristic heuristic as follows.

\begin{proposition}[Proposition 3.3 of Takemura and Kuriki \cite{Takemura-Kuriki02}]
\label{prop:equivalence}
\[
 \E[\chi(A_b)] =
 \PP(b)
\ \ \mbox{for all $b\ge 0$}.
\]
\end{proposition}
Using this, Takemura and Kuriki \cite{Takemura-Kuriki02} provided an alternative proof that the confidence band of Naiman \cite{Naiman86} is conservative.

\section{Main results}
\label{sec:main}

Recall that our aim is to derive the upper tail probability of the maximum of 
the Gaussian random field $Z(x,h)$ defined in (\ref{Z}), or equivalently,
the chi-square random process $Y(x)$ defined in (\ref{Y}).
The theorem below provides an answer, and proof is provided in the Appendix.

\begin{theorem}
\label{thm:main}
Let $\xi\sim \N_n(0,I)$, $n=p(k-1)$.
Let $\Gamma\subset\S^{p-1}$ and $M\subset\S^{n-1}$ be defined by (\ref{Gamma}) and (\ref{M}),
and let $|\Gamma|$ denote the length of $\Gamma$.
Assume Assumption \ref{as:psi} on $\psi$.
Then, as $b\to\infty$,
\begin{align*}
 \Pr\biggl(\max_{(x,h)\in\mathcal{X}\times\S^{k-2}} Z(x,h) \ge b \biggr)
=& \Pr\biggl(\max_{x\in\mathcal{X}} Y(x) \ge b^2 \biggr) \\
=& \Pr\biggl(\max_{u\in M}\xi^\top u \ge b \biggr) \\
=& \PP(b) + O\bigl(b^{n-2} e^{-(1+\tan^2\theta_{\mathrm{c}})b^2/2}\bigr),
\end{align*}
where
\begin{align}
\label{p-hat2}
\PP(b)
= \frac{\Gamma(\frac{k}{2})}{\sqrt{\pi}\,\Gamma(\frac{k-1}{2})} |\Gamma| \bigl\{\G_k(b^2) - \G_{k-2}(b^2) \bigr\}
 + \chi(\Gamma) \G_{k-1}(b^{2}).
\end{align}
Note that if $\Gamma$ (and hence $M$) has no boundary, then $\Gamma$ is homeomorphic to $\S^1$, and therefore $\chi(\Gamma)=0$.
Otherwise, $\chi(\Gamma)$ is the number of connected components of $\Gamma$.
\end{theorem}

\begin{theorem}
\label{thm:conservative}
Assume Assumption \ref{as:psi}.
Suppose that $\Gamma$ has boundaries.
The approximation formula given in Theorem \ref{thm:main} is a conservative bound, specifically,
\begin{align*}
 \Pr\biggl(\max_{u\in M}\xi^\top u \ge b \biggr) \le
 \PP(b)
\ \ \mbox{for all $b\ge 0$}.
\end{align*}
\end{theorem}
\begin{proof}
Arrange the $p(k-1)\times 1$ vector $\xi=(\xi_1^\top,\ldots,\xi_{k-1}^\top)^\top$, and define a $(k-1)\times p$ matrix $\Xi=(\xi_1,\ldots,\xi_{k-1})^\top$.
Let
\begin{align*}
A_b
 =& \{ u \in M \mid \xi^\top u \ge b \}
 =  \{ h\otimes q \mid (q,h) \in \Gamma\times\S^{k-2},\, h^\top \Xi q \ge b \}
 \subset \S^{p(k-1)-1}, \\
\widetilde A_b
 =& \{ (q,h) \in \Gamma\times\S^{k-2} \mid h^\top \Xi q \ge b \}
 \subset \S^{p-1}\times \S^{k-2}, \\
B_b
 =& \{ q \in \Gamma \mid q^\top\Xi^\top\Xi q \ge b^2 \} \subset \S^{p-1}.
\end{align*}
Note that $A_b$ is the excursion set of the Gaussian random field $\xi^\top u$, $u\in M$,
$\widetilde A_b$ is the excursion set of the Gaussian random field $\sum_{i=1}^{k-1} h_i (\xi_i^\top q) = h^\top\Xi q$, $(q,h)\in \Gamma\times\S^{k-2}$, and
$B_b$ is the excursion set of the chi-square random process $\sum_{i=1}^{k-1} (\xi_i^\top q)^2 = q^\top\Xi^\top\Xi q$, $q\in\Gamma$.
We will prove that for each fixed $\xi$,
$\1(A_b\ne\emptyset)=\1(\widetilde A_b\ne\emptyset)=\1(B_b\ne\emptyset)$ and
$\chi(A_b)=\chi(\widetilde A_b)=\chi(B_b)$.

First, note that owing to Assumption \ref{as:psi}, the map $(q,h)\mapsto h\otimes q$ is one-to-one.
Hence, $A_b$ and $\widetilde A_b$ are homeomorphic and therefore $\1(A_b\ne\emptyset)=\1(\widetilde A_b\ne\emptyset)$ and $\chi(A_b)=\chi(\widetilde A_b)$.

Moreover, noting that
$\widetilde A_b\ne\emptyset$ $\Leftrightarrow$ $\max_h h^\top\Xi q\ge b$ for some $q$ $\Leftrightarrow$ $q^\top\Xi\Xi^\top q \ge b^2$ for some $q$ $\Leftrightarrow$ $B_b\ne\emptyset$, that is, $\1(\widetilde A_b\ne\emptyset)=\1(B_b\ne\emptyset)$, we can write
\[
 \widetilde A_b = \bigsqcup_{q\in B_b}\{(q,h) \mid h\in\S^{k-2},\,h^\top\Xi q \ge b \}.
\]
Given $b\ge 0$, the set $\{ h\in\S^{k-2} \mid h^\top\Xi q \ge b \}$ is contractible and star-shaped about the point $h^*(q)=\Xi q/\Vert \Xi q\Vert$.
That is,
the map
\[
 \varphi : \widetilde A_b\times[0,1]\to\widetilde A_b,\quad (q,h,t) \mapsto \biggl(q,\frac{(1-t) h+t h^*(q)}{\Vert (1-t) h+t h^*(q)\Vert}\biggr)
\]
is continuous, and
$\varphi\bigl(\widetilde A_b\times\{0\}\bigr)=\widetilde A_b$ is homotopy equivalent to the set $\varphi\bigl(\widetilde A_b\times\{1\}\bigr)=\bigsqcup_{q\in B_b}\{(q,h^*(q))\}$.
This is homotopy equivalent to $\bigsqcup_{q\in B_b}\{q\} = B_b$.
Hence, $\chi(\widetilde A_b)=\chi(B_b)$.

Recall that $B_b$ is the excursion set of the chi-square random process on the one-dimensional index set $\Gamma$.
This means that $B_b$ is also one-dimensional, and $\chi(B_b)$ is only the number of connected components of $B_b$.
Therefore $\1(B_b\ne\emptyset)\le\chi(B_b)$.
By taking expectations,
\begin{align*}
 \Pr\biggl(\max_{u\in M}\xi^\top u \ge b \biggr) = \E[\1(A_b\ne\emptyset)] & = \E[\1(B_b\ne\emptyset)] \\
& \le \E[\chi(B_b)] = \E[\chi(A_b)] =
 \PP(b).
\end{align*}
The last equality is owing to Proposition \ref{prop:equivalence}.
\end{proof}

\begin{remark}
Naiman \cite{Naiman86} proved that application of the volume-of-tube method to a Gaussian random process with a one-dimensional index set always provides a conservative band.
Theorem \ref{thm:conservative} is a generalization of Naiman's \cite{Naiman86} inequality to a chi-square random process.
\end{remark}

\begin{theorem}
\label{thm:critical_radius}
The interior and boundary of $\Gamma$ are denoted by $\Int\Gamma$ and $\partial\Gamma$, respectively.
The critical radius $\theta_{\mathrm{c}}$ of $M$ is given by
\[
 \tan^2\theta_{\mathrm{c}} = \min\biggl\{ \inf_{x\ne \tilde x,\,\psi(x)\in\Int\Gamma} \frac{(1-\alpha s)^2}{1-s^2-\alpha^2 t^2},
 \inf_{x\ne \tilde x,\,\psi(x)\in\partial\Gamma} \frac{(1-\alpha s)^2}{1-s^2-\max\{0,\varepsilon(x)\alpha t\}^2} \biggr\},
\]
where the infima are taken over $x,\tilde x\in\mathcal{X}$, and $\alpha\in[-1,1]$ as well as additional conditions (arguments of\/ $\inf$), and
\[
 s=s(x,\tilde x)=\psi(x)^\top\psi(\tilde x), \quad
 t=t(x,\tilde x)=\frac{\psi_x(x)^\top\psi(\tilde x)}{\Vert\psi_x(x)\Vert},
\]
$\psi_x(x)=\partial\psi(x)/\partial x$,
\[
\varepsilon(x) = \begin{cases}
 1 & (\mbox{$\psi_x(x)$ is inward to $\Gamma$}), \\
 -1 & (\mbox{$\psi_x(x)$ is outward to $\Gamma$}).
 \end{cases}
\]
$\psi_x(x)$ is said to be inward or outward to $\Gamma$ if
the support cone of $\Gamma$ at $\psi(x)$ is
$S_{\psi(x)}\Gamma=\{\lambda\psi_x(s) \mid \lambda\ge 0\}$ or 
$\{\lambda\psi_x(s) \mid \lambda\le 0\}$, respectively.
\end{theorem}

\begin{theorem}
\label{thm:local_critical_radius}
Assume Assumption \ref{as:psi}.
Moreover, assume that $\psi:\mathcal{X}\to\S^{p-1}$ is of $C^4$-class.
Then, the local critical radius $\theta_{\mathrm{c,loc}}$ is given by
\[
\tan^{2}\theta_{\mathrm{c,loc}} = \min\biggl\{
 \inf_{x\in\mathcal{X}:\kappa(x)\le 2} \biggl\{1-\frac{\kappa(x)}{4}\biggr\},
 \inf_{x\in\mathcal{X}:\kappa(x)\ge 2} \frac{1}{\kappa(x)} \biggr\}
\]
with
\begin{equation}
\label{kappa}
 \kappa(x) = \frac{\psi_{xx}(x)^\top\psi_{xx}(x)}{\{\psi_x(x)^\top\psi_x(x)\}^2}
  -\frac{\{\psi_{xx}(x)^\top\psi_x(x)\}^2}{\{\psi_x(x)^\top\psi_x(x)\}^3} -1,
\end{equation}
where $\psi_x(x)=\partial\psi(x)/\partial x$ and $\psi_{xx}(x)=\partial^2\psi(x)/\partial x^2$.
\end{theorem}

The proofs of Theorems \ref{thm:critical_radius} and \ref{thm:local_critical_radius} are included in the Appendix.

\subsubsection*{A numerical example}

At the end of this section, we provide a numerical example
to determine the accuracy of the approximation formula given in Theorem \ref{thm:main}, and degree of conservativeness proved by Theorem \ref{thm:conservative}.

Suppose that $f(x)=(1,x,x^2)^\top$, $\mathcal{X}=[-1,1]$, and
\[
 \Sigma = \begin{pmatrix} 1 & 0 & \frac{2}{3} \\ 0 & \frac{2}{3} & 0 \\ \frac{2}{3} & 0 & 1 \end{pmatrix}, \quad
 \Sigma^{1/2} = \begin{pmatrix}
 1 & 0 & \frac{2}{3} \\ 0 & \sqrt{\frac{2}{3}} & 0 \\ 0 & 0 & \frac{\sqrt{5}}{3} \end{pmatrix}.
\]
Then,
\begin{align*}
%
& \psi(x) = \frac{1}{3(1+x^2)} \bigl(3+2x^2,\sqrt{6}x,\sqrt{5}x^2\bigr)^\top, \quad
 |\Gamma| = \int_{\mathcal{X}} \Vert\dot\psi(x)\Vert \,\dd x = 
 \int_{-1}^1 \sqrt{\frac{2}{3}}\frac{1}{1+x^2} \,\dd x = \frac{\pi}{\sqrt{6}}.
\end{align*}
$\kappa(x)$ in (\ref{kappa}) is always $5$.
Hence, the local critical radius is $\theta_{\mathrm{c,loc}}=\tan^{-1}(1/\sqrt{5})=0.134 \pi$. 
Further, we can also confirm that the critical radius is the same as $\theta_{\mathrm{c}}=\theta_{\mathrm{c,loc}}$ using Mathematica \cite{Mathematica}.

Under this setting, we suppose the case of $k=3$.
The probability we need is
\begin{equation}
\label{maxY}
 \Pr\biggl( \max_{x\in[-1,1]}Y(x)\ge b^2 \biggr)
 = 1-\Pr\bigl(T(x,c)\le b,\ \forall x\in[-1,1],\,\forall c\in\mathcal{C}\bigr),
\end{equation}
where
\[
 Y(x)=\sum_{i=1}^2 \{\xi_i^\top\psi(x)\}^2,\ \ \xi_1,\xi_2\sim \N_3(0,I)\ \rm i.i.d.
\]
is a chi-square random process $Y(x)$ with two degrees of freedom.
The tube formula for the upper tail probability (\ref{maxY}) is
\begin{equation}
\label{maxY-hat}
 \PP(b)
= \frac{\pi}{2\sqrt{6}} \bigl\{\G_3(b^2)-\G_1(b^2)\bigr\} + \G_2(b^2)
= \biggl(\frac{\sqrt{\pi}}{2\sqrt{3}} b +1\biggr) e^{-b^2/2}.
\end{equation}
Figure \ref{fig:tail_prob} depicts the upper tail probability of the maximum (\ref{maxY}) and its approximate value (\ref{maxY-hat}).
We can see that the tube formula approximates the true upper tail probability with sufficient accuracy in the moderate tail regions (for example, the upper probability is less than 0.2),
and it provides a conservative bound as per Theorem \ref{thm:conservative}.

\begin{figure}[h]
\begin{center}
\includegraphics[width=0.65\linewidth]{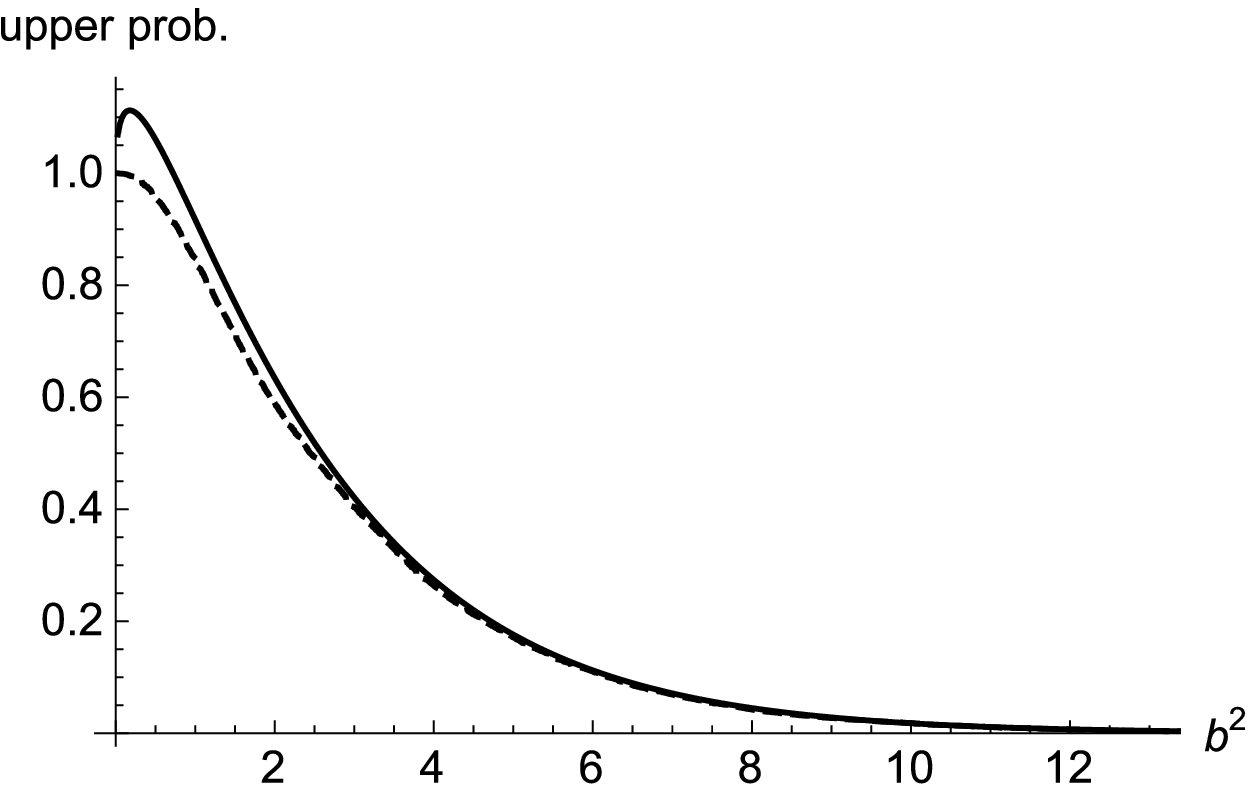} 
\caption{Upper tail probability of the maximum of chi-square process $Y(x)$.}

\small
\vspace*{1mm}
(solid line: tube formula, dashed line: Monte Carlo with 10,000 replications)

\label{fig:tail_prob}
\end{center}
\end{figure}

We have proposed that the threshold for the confidence band should be determined as the solution $b=b_{\mathrm{tube},1-\alpha}$ for $\PP(b)=\alpha$.
Figure \ref{fig:confidence_coeff} depicts the actual confidence coefficient (coverage probability) 
\[
 \Pr\biggl( \max_{x\in[-1,1]}Y(x)\ge b_{\mathrm{tube},1-\alpha}^2 \biggr), \ \ \alpha\in[0,1].
\]
This further demonstrates that the confidence bands obtained by the tube method are always conservative and very accurate.

\begin{figure}[h]
\begin{center}
\includegraphics[width=0.65\linewidth]{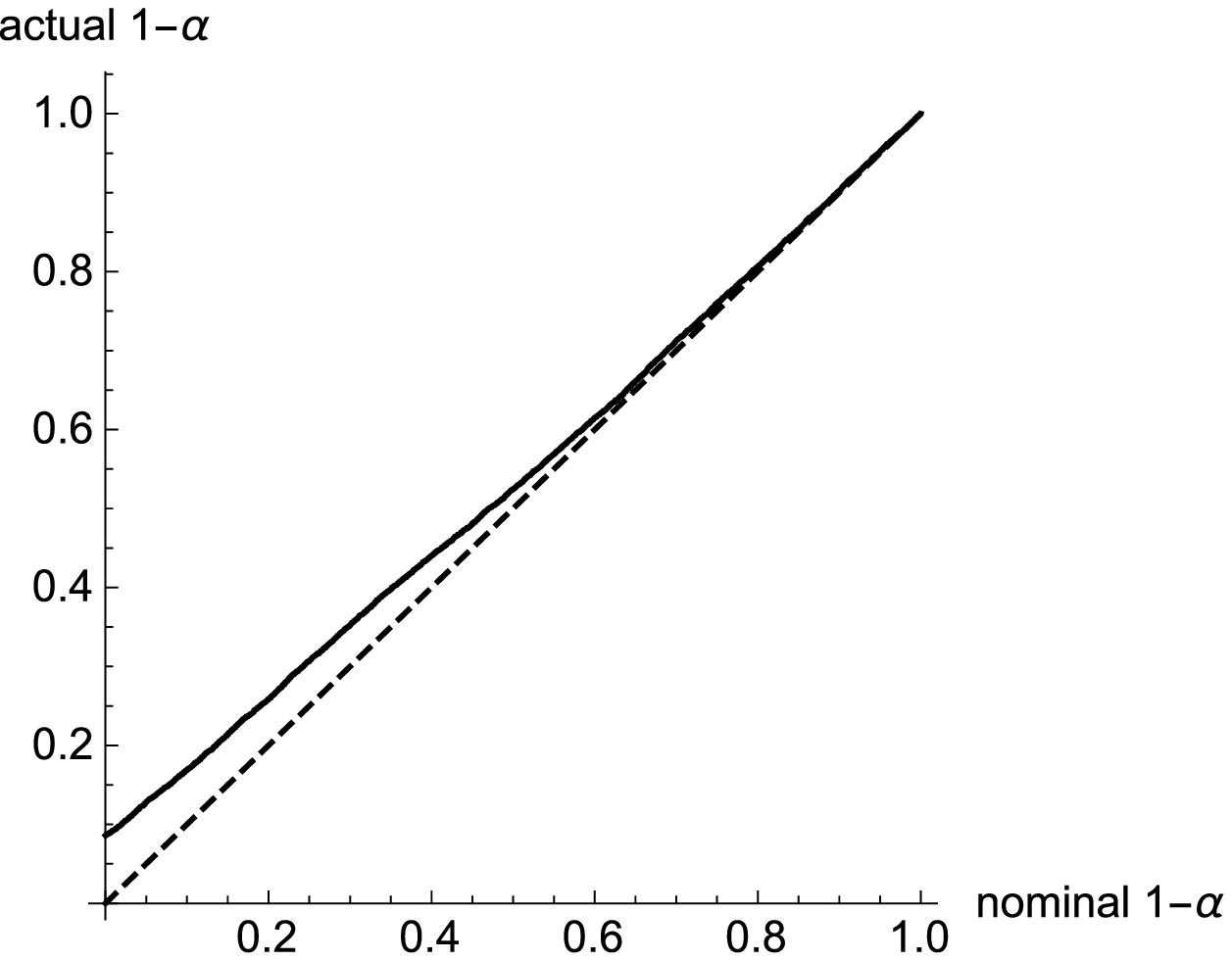}
\caption{Nominal confidence coefficient vs.\ Actual confidence coefficient.}

\small
\vspace*{1mm}
\footnotesize
(solid line: actual confidence coefficient, dashed line: 45-degree line)

\label{fig:confidence_coeff}
\end{center}
\end{figure}

\section{Simulation study under model misspecification}
\label{sec:simulation}

Throughout this paper, it is assumed that the nonlinear model has a finite number of basis functions $g_i(x)=\beta_i^\top f(x)$ in (\ref{beta_f}).
However, we can only approximate the true model in practice.
Under a slight misspecification of the model, Sun and Loader \cite{Sun-Loader94} estimated the bias of the coverage probability, and proposed an adjustment to the volume-of-tube formula.
Although their approach may be applied to our model, the result would be more complicated.
Instead, to investigate what happens under model misspecification, we conducted a Monte Carlo simulation study in the following setting.

The domain of explanatory variable is set to be $\mathcal{X}=[0,1]$.
The data are generated from the model
\[
 y_{ij} = g_i(x_j) + \varepsilon_{ij}, \ \ \varepsilon_{ij}\sim \N(0,1), \ \ i=1,\ldots,k,\ \ j=1,\ldots,n,
\]
where $k=3$, $n=11$, and $x_j=(j-1)/n$, $j=1,\ldots,n$.
As the true regression curve is $g_i(x)$, we assume three models.

Model 1:
\[
 g_i(x) = \beta_i^\top f_{2,5,0,1}(x), \quad \beta_1=(0,\ldots,0)^\top, \ \ \beta_2= K(0,0,1/2,1,1)^\top, \ \ \beta_3= K(0,0,4/3,0,0)^\top,
\]
where $K=1$, $3$, or $9$,
\[
 f_{d,m,a,b} = \left( B_d\biggl(\frac{x-a}{b-a}(m-d)-(i-d-1)\biggr) \right)_{i=1,\ldots,m},
\]
and $B_d(\cdot)$ is the B-spline function
\begin{equation}
\label{Bd}
 B_d(x) = \sum_{r=0}^{d+1} (-1)^{d+1-r} {d+1 \choose r} \frac{(r-x)_+^d}{d!}
\end{equation}
(\cite{deBoor78}, p.\,89).

Model 2:
\[
 g_1(x) = 0, \quad g_2(x) = K\sin(x\pi/2), \quad g_3(x) = K\sin(x\pi), \quad K=1,3,9.
\]

Model 3:
\[
 g_1(x) = 0, \quad g_2(x) = K\frac{e^{-x/2}-e^{-x}}{e^{-1/2}-e^{-1}}, \quad g_3(x) = K\frac{\cosh(x-1/2)-1}{\cosh(1/2)-1}, \quad K=1,3,9.
\]
For all models, $g_2(x)$ is unimodal, and $g_3(x)$ is increasing.
$g_2(x)$ and $g_3(x)$ are designed to have the range $[0,K]$.

We fit the curve $\beta_i^\top f_{2,m,0,1}(x)$ to the generated data $y_{ij}$, where $m=3,\ldots,10$.
Using these models, we constructed a $1-\alpha=0.95$ confidence band.
Coverage probabilities were estimated based on Monte Carlo simulations with 1,000,000 replications, and are summarized in Table \ref{tab:coverage_prob}.
In this table,
\begin{equation}
\label{delta}
 \delta = \max_{x\in\mathcal{X},\,c\in\mathcal{C}}
\left| \frac{\sum_{i=1}^k c_i \{(\beta_i^*)^\top f_{2,m,0,1}(x) - g_i(x)\}}{\sqrt{f_{2,m,0,1}(x)^\top\Sigma f_{2,m,0,1}(x)}} \right|,
\quad
\Sigma = \left(\sum_{i=1}^n f_{2,m,0,1}(x_i)^\top f_{2,m,0,1}(x_i)\right)^{-1}
\end{equation}
is the bias of regression function,
where $\beta_i^*$ is the best parameter in the assumed model $\beta_i^\top f_{2,m,0,1}(x)$.
\begin{equation}
\label{Delta}
 \Delta = \max\bigl\{ \alpha-\PP(b_{\mathrm{tube},1-\alpha}+\delta), \PP(b_{\mathrm{tube},1-\alpha}-\delta)-\alpha \bigr\}
\end{equation}
is an approximate upper bound of the bias of coverage probability,
where $b_{\mathrm{tube},1-\alpha}$ is the approximate value of $b_{1-\alpha}$ obtained by the tube method.  (See \ref{sec:bound} for the detail.)
 
\begin{table}
\begin{center}
\caption{Coverage probability under model misspecification ($1-\alpha=0.95$)}
\label{tab:coverage_prob}
{\footnotesize (prob: coverage probability, $\delta$: bias (\ref{delta}), $\Delta$: bound for coverage probability bias (\ref{Delta}))}

\bigskip
\begingroup
\renewcommand{\arraystretch}{1.1}
\begin{small}
\begin{tabular}{cccccccccccccc}
\hline
\multirow{2}{2ex}{$m$} && \multicolumn{3}{c}{Model 1 ($K=1$)} && \multicolumn{3}{c}{Model 1 ($K=3$)} && \multicolumn{3}{c}{Model 1 ($K=9$)} \\
\cline{3-5} \cline{7-9} \cline{11-13}
&& prob & $\delta$ & $\Delta$ && prob & $\delta$ & $\Delta$ && prob & $\delta$ & $\Delta$ \\
\hline
3 && 0.9365 & 0.4692 & 0.1155 && 0.7872 & 1.4076 & 0.8240 && 0.0000 & 4.2227 & 1.3542 \\
4 && 0.9422 & 0.3996 & 0.0965 && 0.8509 & 1.1987 & 0.6909 && 0.0076 & 3.5961 & 1.6907 \\
5 && 0.9512 & 0.0000 & 0.0000 && 0.9512 & 0.0000 & 0.0000 && 0.9512 & 0.0000 & 0.0000 \\
6 && 0.9511 & 0.1006 & 0.0176 && 0.9477 & 0.3018 & 0.0694 && 0.9111 & 0.9053 & 0.4550 \\
7 && 0.9514 & 0.0448 & 0.0074 && 0.9509 & 0.1343 & 0.0251 && 0.9465 & 0.4030 & 0.1100 \\
8 && 0.9515 & 0.0000 & 0.0000 && 0.9515 & 0.0000 & 0.0000 && 0.9515 & 0.0000 & 0.0000 \\
9 && 0.9515 & 0.0175 & 0.0029 && 0.9514 & 0.0526 & 0.0090 && 0.9504 & 0.1578 & 0.0316 \\
10&& 0.9516 & 0.0218 & 0.0036 && 0.9515 & 0.0653 & 0.0116 && 0.9508 & 0.1959 & 0.0421 \\
\hline
\end{tabular}

\bigskip

\begin{tabular}{cccccccccccccc}
\hline
\multirow{2}{2ex}{$m$} && \multicolumn{3}{c}{Model 2 ($K=1$)} && \multicolumn{3}{c}{Model 2 ($K=3$)} && \multicolumn{3}{c}{Model 2 ($K=9$)} \\
\cline{3-5} \cline{7-9} \cline{11-13}
&& prob & $\delta$ & $\Delta$ && prob & $\delta$ & $\Delta$ && prob & $\delta$ & $\Delta$ \\
\hline
3 && 0.9509 & 0.06491 & 0.0099 && 0.9486 & 0.1947 & 0.0346 && 0.9277 & 0.5842 & 0.1640 \\
4 && 0.9511 & 0.04999 & 0.0077 && 0.9498 & 0.1500 & 0.0264 && 0.9374 & 0.4499 & 0.1157 \\
5 && 0.9512 & 0.01271 & 0.0019 && 0.9511 & 0.0381 & 0.0060 && 0.9504 & 0.1143 & 0.0199 \\
6 && 0.9516 & 0.00494 & 0.0008 && 0.9516 & 0.0148 & 0.0023 && 0.9515 & 0.0445 & 0.0072 \\
7 && 0.9514 & 0.00234 & 0.0004 && 0.9514 & 0.0070 & 0.0011 && 0.9514 & 0.0211 & 0.0034 \\
8 && 0.9515 & 0.00137 & 0.0002 && 0.9515 & 0.0041 & 0.0007 && 0.9515 & 0.0123 & 0.0020 \\
9 && 0.9515 & 0.00119 & 0.0002 && 0.9515 & 0.0036 & 0.0006 && 0.9515 & 0.0107 & 0.0017 \\
10&& 0.9516 & 0.00076 & 0.0001 && 0.9516 & 0.0023 & 0.0004 && 0.9516 & 0.0068 & 0.0011 \\
\hline
\end{tabular}

\bigskip

\begin{tabular}{cccccccccccccc}
\hline
\multirow{2}{2ex}{$m$} && \multicolumn{3}{c}{Model 3 ($K=1$)} && \multicolumn{3}{c}{Model 3 ($K=3$)} && \multicolumn{3}{c}{Model 3 ($K=9$)} \\
\cline{3-5} \cline{7-9} \cline{11-13}
&& prob & $\delta$ & $\Delta$ && prob & $\delta$ & $\Delta$ && prob & $\delta$ & $\Delta$ \\
\hline
3 && 0.9512 & 0.02384 & 0.0034 && 0.9508 & 0.07152 & 0.0109 && 0.9483 & 0.2146 & 0.0390 \\
4 && 0.9513 & 0.00766 & 0.0011 && 0.9512 & 0.02299 & 0.0034 && 0.9511 & 0.0690 & 0.0109 \\
5 && 0.9512 & 0.00218 & 0.0003 && 0.9512 & 0.00653 & 0.0010 && 0.9512 & 0.0196 & 0.0030 \\
6 && 0.9516 & 0.00095 & 0.0002 && 0.9516 & 0.00285 & 0.0004 && 0.9516 & 0.0086 & 0.0013 \\
7 && 0.9514 & 0.00046 & 0.0001 && 0.9514 & 0.00138 & 0.0002 && 0.9514 & 0.0042 & 0.0006 \\
8 && 0.9515 & 0.00028 & 0.0000 && 0.9514 & 0.00083 & 0.0001 && 0.9515 & 0.0025 & 0.0004 \\
9 && 0.9515 & 0.00024 & 0.0000 && 0.9515 & 0.00071 & 0.0001 && 0.9515 & 0.0021 & 0.0003 \\
10&& 0.9516 & 0.00014 & 0.0000 && 0.9516 & 0.00042 & 0.0001 && 0.9516 & 0.0013 & 0.0002 \\
\hline
\end{tabular}
\end{small}
\endgroup
\smallskip

\end{center}
\end{table}
\medskip                                                                       

\begin{table}
\begin{center}
\caption{Average band-width $W$ ($1-\alpha=0.95$)}
\label{tab:width}
\bigskip
\begingroup
\renewcommand{\arraystretch}{1.2}
\begin{small}
\begin{tabular}{cccccccccc}
\hline
$m$ && 3 & 4 & 5 & 6 & 7 & 8 & 9 & 10 \\ 
\hline
$W$ && 1.463 & 1.752 & 2.017 & 2.275 & 2.546 & 2.764 & 2.990 & 3.211 \\
\hline
\end{tabular}
\end{small}
\endgroup
\end{center}
\end{table}

From this table, we first see that, for the true models ($m=5,8$ when model 1 is true), the coverage probabilities are more than, but approximately equal to, the nominal value 0.95, meaning that the proposed method is valid.
The most remarkable point is that, throughout the study, the coverage probabilities are kept at approximately 0.95, unless the assumed model is too small, and the bias $\delta$ is large.

Table \ref{tab:width} shows the average width of the confidence band defined by
\[
 W = \frac{\int_{\mathcal{X}} b_{\mathrm{tube},1-\alpha} \sqrt{f(x)^\top\Sigma f(x)} \,\dd x}{\int_{\mathcal{X}} \,\dd x} = b_{\mathrm{tube},0.95} \int_0^1 \sqrt{f_{2,m,0,1}(x)^\top\Sigma f_{2,m,0,1}(x)} \,\dd x.
\]
When the model is increasing in size, $W$ is increasing in size.
This suggests that a smaller model is preferable, unless it is too small to cause serious bias.

In summary, too small of a model should surely be avoided, whereas, a larger model has the disadvantage of having a wider confidence band.
This trade-off is crucially important in practice, and a promising future research topic, although it is out of scope for this paper.
For related topics, refer to Casella and Hwang \cite{Casella-Hwang12}, for shrinkage confidence bands, and Leeb et al.\ \cite{Leeb-etal15}, for confidence band post-model selection.

\section{Growth curve analysis} 
\label{sec:growth}

\newcommand{\MSM}{$^{\rm MSM}$}

As mentioned in Section \ref{sec:introduction}, the growth curve analysis is one of our research objectives to which we apply our method.
In this section, we demonstrate the analysis of mouse growth as an illustration.
Sun et al.\ \cite{Sun-Raz-Faraway99} proposed simultaneous confidence bands for a growth curve by virtue of the volume-of-tube method.
Differently from their analysis, we focus on the contrast of several growth curves.

Mice are one of the most popular model organisms, and are often used in genomic research.
Figure \ref{fig:growth} depicts the average body weights of male mice from four different strains measured from 2 to 20 weeks after birth.
The four strains are C57BL/6 (referred to as B6), MSM/Ms (MSM), B6-Chr17\MSM (B6-17), and B6-ChrXT\MSM (B6-XT).
Among these, B6 is the most common laboratory strain and serves as the standard.
MSM is a wild-derived strain having contrasting properties to B6 such as non-black color, small size, and aggressive behavior.
B6-17 and B6-XT are artificial strains known as consomic mice made from B6 and MSM.
B6-17 has all the chromosomes from B6, and only chromosome 17 from MSM; B6-XT has all the chromosomes from B6, and only half of the X chromosome from MSM.
By comparing the consomic strains with B6, we expect to reveal the role of each chromosome.

\begin{figure}[h]
\begin{center}
\begin{tabular}{ccc}
\includegraphics[width=0.8\linewidth]{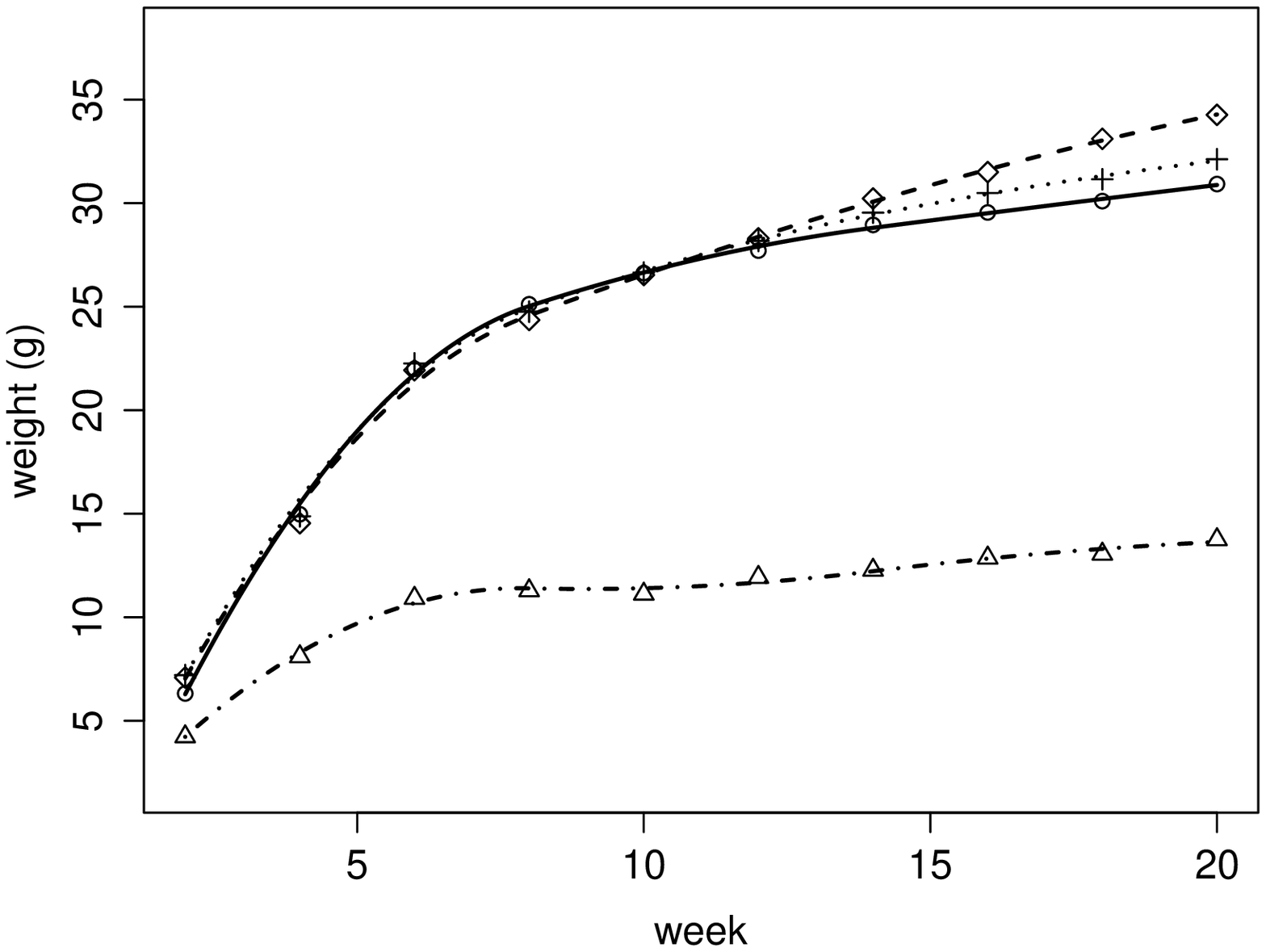}
\end{tabular}
\caption{Average body weights of mice from four strains.}

\medskip
{\footnotesize
(sample mean: $\circ$ (B6), $+$ (B6-17), $\diamond$ (B6-XT), {\tiny\mbox{$\triangle$}} (MSM); \hspace*{10mm} \\
fitted curve: {\bf ---} (B6), {\bf $\cdots$} (B6-17), {\bf --\,--} (B6-XT), {\bf --\,$\cdot$\,--} (MSM))
}
\label{fig:growth}
\end{center}
\end{figure}

The dataset we utilized is publicly available as Supplemental Table S1 of Takada et al.\ \cite{Takada-etal08}.
In their experiments, the weight (unit: gram) $y_{ijh}$ of the $h$th individual from strain $i$ was measured at time point $x_j$. 
The measurement time points were $\{x_1,\ldots,x_{10}\}=\{2,4,\ldots,20\}$ ($n=10$).
This dataset includes the average body weight $y_{ij}$ of strain $i$ at time $x_j$, and its standard error
\[
 y_{ij} = \frac{1}{r_i} \sum_{h=1}^{r_i} y_{ij h}, \quad
 \widehat{\mathrm{s.e.}}(y_{ij}) = \sqrt{\frac{1}{r_i^2}\sum_{h=1}^{r_i} (y_{ij h}-y_{ij})^2},
\]
as well as the number $r_i$ of individuals of strain $i$.

In the following analysis, we use $k=3$ groups (strains) B6 ($i=1$), B6-17 ($i=2$), and B6-XT ($i=3$).
The number of individuals are $r_1=12$, $r_2=24$, and $r_3=12$. 

We fit the model (\ref{model}) to these data.
We estimate the variance as
\[
 \widehat\sigma(x_j)^2
 = \frac{1}{\sum_{i=1}^k (r_i-1)}\sum_{i=1}^k \sum_{h=1}^{r_i}(y_{ijh}-y_{ij})^2
 = \frac{1}{\sum_{i=1}^k (r_i-1)}\sum_{i=1}^k r_i^2\widehat{\mathrm{s.e.}}(y_{ij})^2,
\]
which is used as the true value $\sigma(x_j)^2$ hereafter.
Figure \ref{fig:se} plots the estimated standard error $\widehat\sigma(x_j)$.
One particular feature of this dataset is that the experiment is well controlled and measurement errors are quite small.

\begin{figure}[h]
\begin{center}
\includegraphics[width=0.65\linewidth]{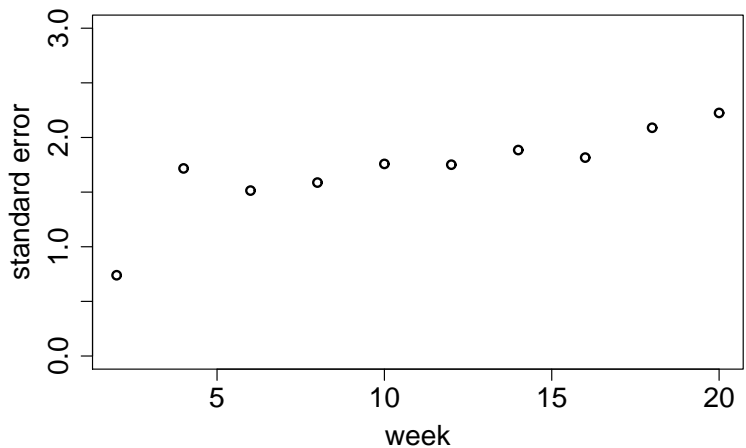}
\caption{Estimated standard error $\widehat\sigma(x_j)$.}
\label{fig:se}
\end{center}
\end{figure}

As the basis function $f(x)$,
we consider a family of basis functions
\[
 f(x) = f_{d,m,2,20}(x) =
 \left(B_d\biggl(\frac{x-2}{20-2}(m-d)-(i-d-1)\biggr)\right)_{1\le i\le m},
\]
with $B_d(x)$ given in (\ref{Bd}).
$f_{d,m,2,20}(x)$ consists of $m$ B-spline bases with equally-spaced knots at intervals of $(20-2)/(m-d)$.
Note that $f_{d,m,2,20}(x)$ is piecewise of class $C^{d}$.

In the range $d=2,3,4$ and $m=d+1,d+2,\ldots,n\,(=10)$, we searched for the best model that minimizes AIC and BIC defined below:
\[
 \mathrm{AIC}_{d,m} = L_{d,m} + 2 k m, \quad
 \mathrm{BIC}_{d,m} = L_{d,m} + \sum_{i=1}^k \ln(r_i n) m, \quad
 L_{d,m} = \sum_{i=1}^k r_i \sum_{j=1}^n \frac{(y_{ij}-\widehat y_{ij})^2}{\sigma(x_j)^2}
\]
with $k=3$, $n=10$, where $\widehat y_{ij}=\widehat\beta_i^\top f_{d,m,2,20}(x_j)$.
In both criteria, the minimizer was $(d,m)=(2,5)$, which we use as the true value hereafter.

Suppose that we are interested in the period $\mathcal{X}=[a,b]=[2,20]$.
An approximate value of the length of $\Gamma$ in (\ref{Gamma}) is given by
\[
 |\Gamma| \approx
 \sum_{t=1}^{N} \bigl\Vert\psi(x_{t})-\psi(x_{t-1})\bigr\Vert,
\]
where $x_t = a+t(b-a)/N$, $t=0,1,\ldots,N$.
When $N=10,000$, the approximate value of $|\Gamma|$ is
$6.989=2.225 \pi$. 
Using this, the critical value is
$b_{1-\alpha}=3.258$ ($\alpha=0.05$). 

To compare $k$ groups, various types of contrasts are used.
For a pairwise comparison between group $i$ and group $j$, we choose
$c = (\ldots,0,\underset{i\rm th}{1},0,\ldots,0,\underset{j\rm th}{-1},0,\ldots)$.
For the comparison of groups $\{i,j\}$ and group $k$, we use
\[
 c = \biggl(\ldots,0,\underset{i\rm th}{\frac{r_i}{r_i+r_j}},0,\ldots,0,\underset{j\rm th}{\frac{r_j}{r_i+r_j}},0,\ldots,0,\underset{k\rm th}{-1},0,\ldots\biggr).
\]
Figure \ref{fig:difference} depicts the difference curves of strains B6-17 vs.\ B6 (left) and B6-XT vs.\ B6 (right), and their 95\% simultaneous confidence bands.
In the left panel, the horizontal line representing zero difference is almost between the confidence bands.
This indicates that there is no significant difference between B6-17 and B6.
In contrast, in the right panel, after around week 14, the horizontal line is outside the confidence bands, thereby indicating that B6-XT and B6 are different during this period.

\begin{figure}[h]
\begin{center}
\begin{tabular}{cc}
\includegraphics[width=0.45\linewidth]{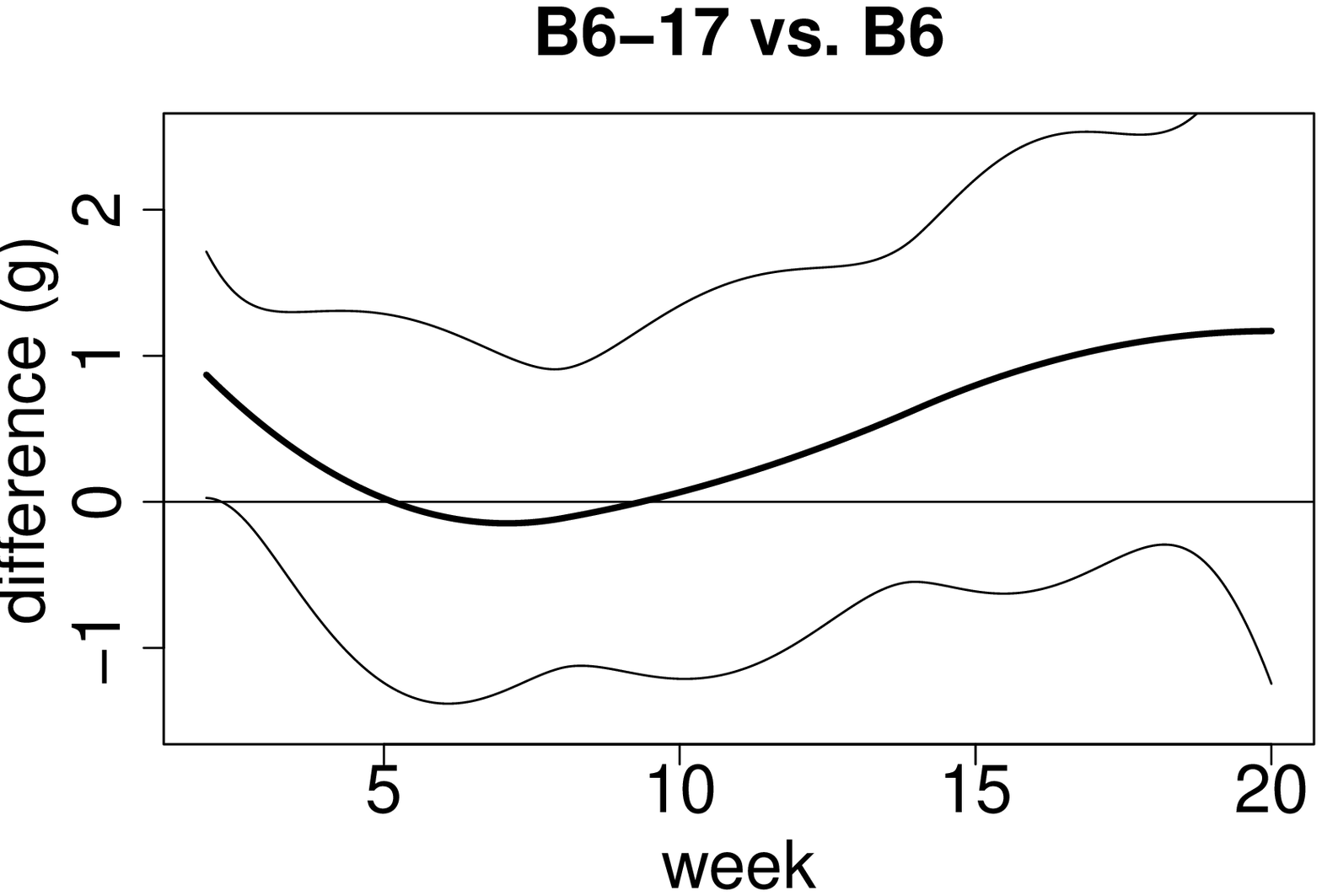} &
\includegraphics[width=0.45\linewidth]{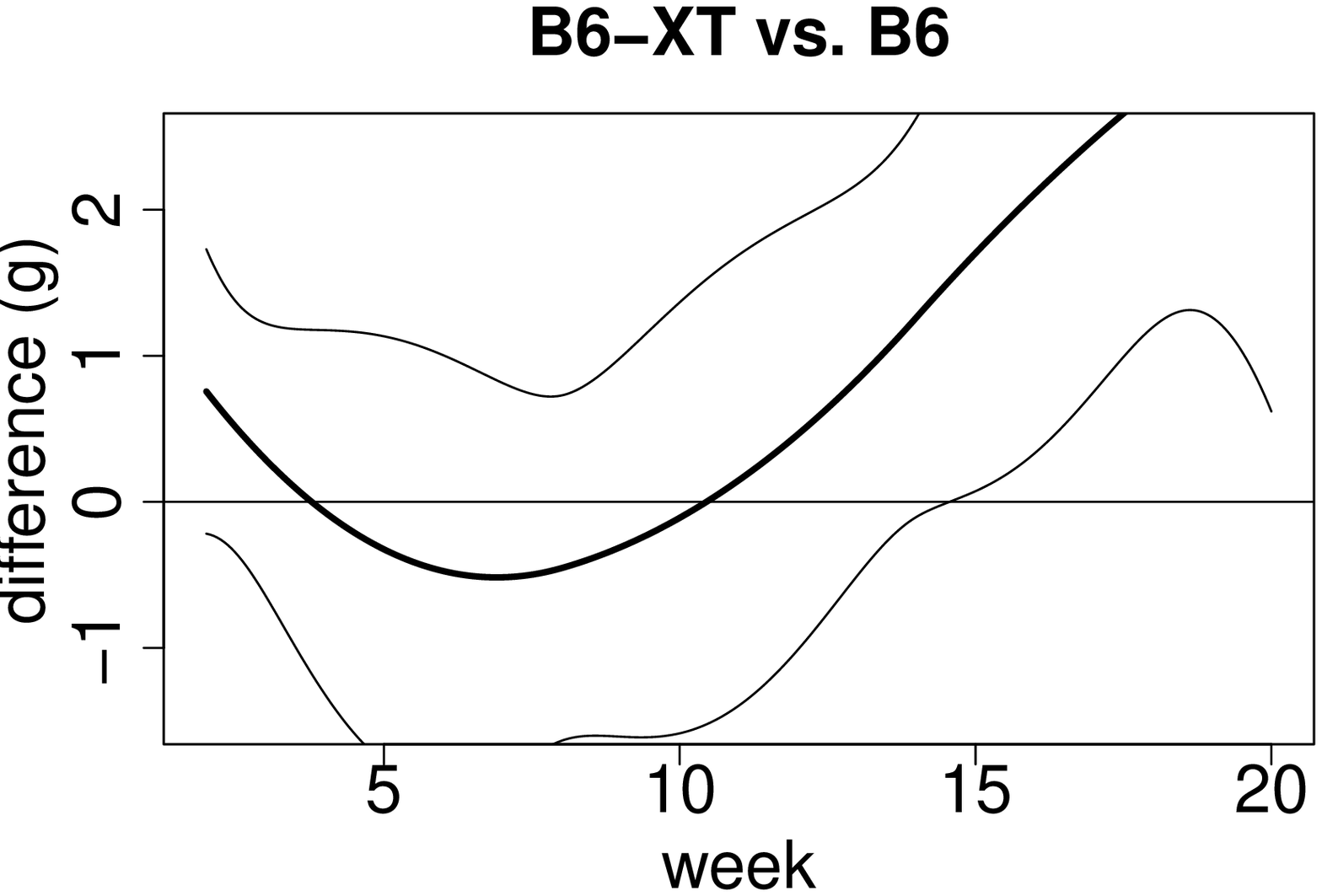}
\end{tabular}
\caption{Differences of body weights and 95\% confidence bands.}
\label{fig:difference}
\end{center}
\end{figure}

For a fixed $x$, the test statistic for the null hypothesis
$H_{0,x}:\beta_1^\top f(x) = \ldots = \beta_k^\top f(x)$ is
\[
 \chi^2(x) = \frac{1}{f(x)^\top\Sigma f(x)}\sum_{i=1}^k r_i \Biggl\{\widehat\beta_i^\top f(x)-\frac{\sum_{i=1}^k r_i\widehat\beta_i^\top f(x)}{\sum_{i=1}^k r_i}\Biggr\}^2.
\]
For a fixed $x$, the null distribution is the chi-square distribution with $k-1$ degrees of freedom.
However, for the overall null hypothesis
$H_0:\beta_1^\top f(x) = \ldots = \beta_k^\top f(x)$ for all $x\in\mathcal{X}$,
the distribution of the maximum of the chi-square random process should be used.
Figure \ref{fig:chi2_process} shows $\chi^2(x)$ and its upper 5\% critical value $b_{0.95}^2$.
As already shown in Figure \ref{fig:difference}, after around week 14, the hypothesis of equality is rejected.

\begin{figure}[h]
\begin{center}
\includegraphics[width=0.65\linewidth]{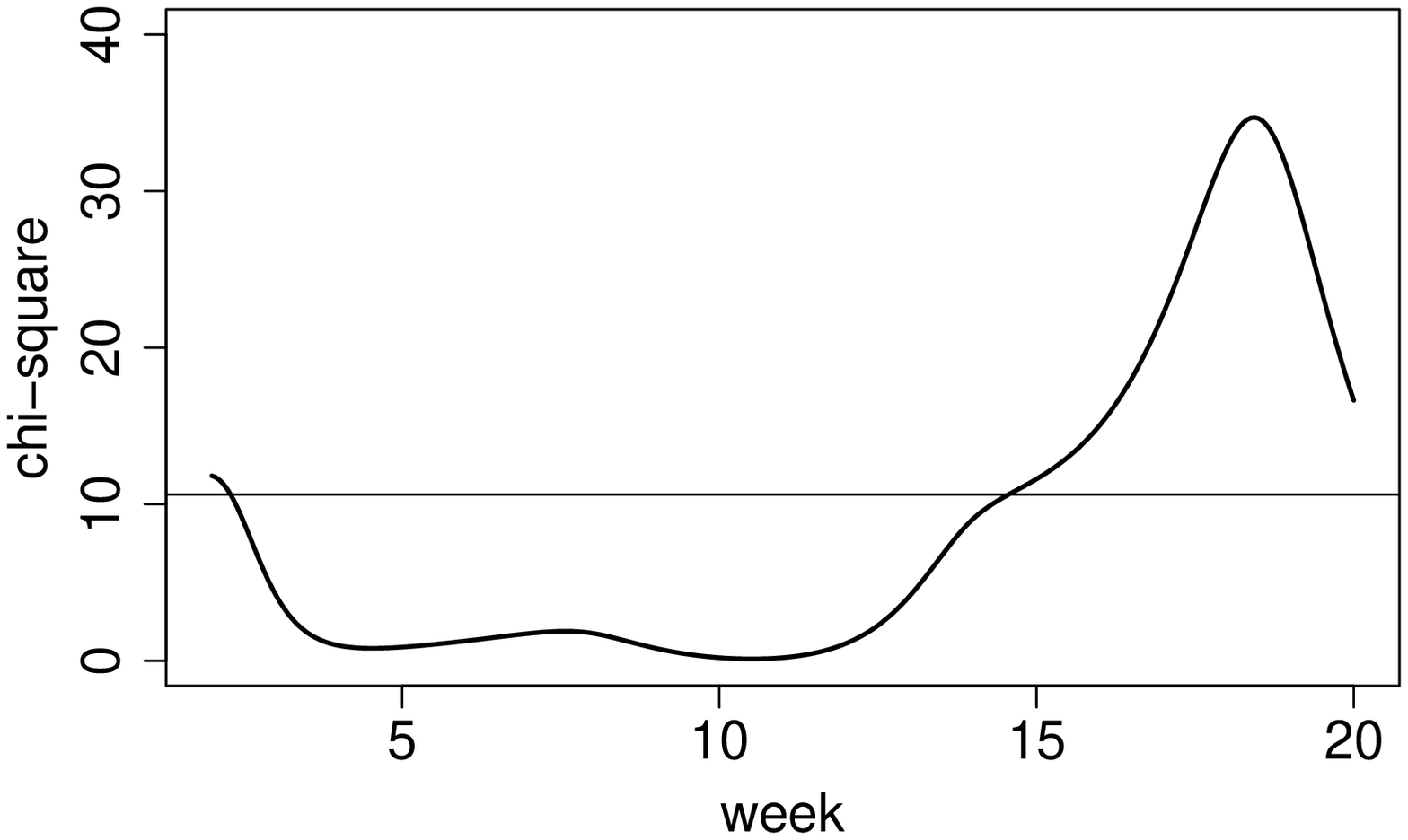}
\caption{Chi-square process $\chi^2(x)$ and its upper 5\% critical value.}
\label{fig:chi2_process}
\end{center}
\end{figure}

\appendix
\section{Appendix: Proofs}

\subsection{Proof of Theorem \ref{thm:main}}

\subsubsection*{Contribution of the inner points $\Int\,M$}

Here, we obtain the coefficients $w_{d+1-e}$ in (\ref{w}) when $M$ is given in (\ref{M}).

Let $h=h(\theta)$, $\theta=(\theta_i)_{1\le k-2}$, be a local coordinate system of $\S^{k-2}$.
For example,
\[
 h = h(\theta) = \left(\begin{array}{l}
 \cos\theta_1 \\
 \sin\theta_1 \cos\theta_2 \\
 \sin\theta_1 \sin\theta_2 \cos\theta_3 \\
 \quad \vdots \\
 \sin\theta_1 \cdots \sin\theta_{k-3} \cos\theta_{k-2} \\
 \sin\theta_1 \cdots \sin\theta_{k-3} \sin\theta_{k-2}
\end{array}\right)_{(k-1)\times 1},
\]
where
\[
 \theta\in\Theta =
\{(\theta_1,\ldots,\theta_{k-2}) \mid 0\le\theta_i\le \pi\ (i=1,\dots,k-3),\ 0\le\theta_{k-2}< 2\pi \}.
\]

Let $(x,\theta)\in\mathcal{X}\times\Theta$ be fixed, and
let $\phi(x,\theta) = h(\theta)\otimes\psi(x) \in M$.
We write $\psi=\psi(x)$, $h=h(\theta)$ and $\phi=\phi(x,\theta)$ for simplicity.
We first assume that $x\in\Int\mathcal{X}$, hence, $\phi(x,\theta)\in\Int M$.

By applying the Gram-Schmidt orthonormalization to the sequence
$\psi, \partial\psi/\partial x, \partial^2\psi/\partial x^2, \ldots$,
we construct the orthonormal basis (ONB)
$\psi_{(i)}$, $i=0,\ldots,p-1$, of $\R^p$.
The first three bases are
\[
 \psi_{(0)} = \psi,\ \ %
 \psi_{(1)} = \frac{1}{\sqrt{g}}\frac{\partial\psi}{\partial x}, \ \ %
 \psi_{(2)} = \frac{1}{\sqrt{\eta-\frac{\gamma^2}{g}-g^2}} \biggl(\frac{\partial^2\psi}{\partial x^2} + g \psi - \frac{\gamma}{g}\frac{\partial\psi}{\partial x}\biggr),
\]
where
\[
 g = g(x) = \biggl(\frac{\partial\psi}{\partial x}\biggr)^\top\biggl(\frac{\partial\psi}{\partial x}\biggr), \quad
 \gamma = \gamma(x) = \biggl(\frac{\partial^2\psi}{\partial x^2}\biggr)^\top\biggl(\frac{\partial\psi}{\partial x}\biggr), \quad
 \eta = \eta(x) = \biggl(\frac{\partial^2\psi}{\partial x^2}\biggr)^\top\biggl(\frac{\partial^2\psi}{\partial x^2}\biggr).
\]
Similarly, from the sequence $h$, $\partial h/\partial\theta_i$, $i=1,\ldots,k-2$, we obtain ONB $h_{(i)}$, $i=0,\ldots,k-2$, of $\R^{k-1}$.
We prepare a $(k-2)\times (k-2)$ upper triangle matrix $D$ such that
\[
 \biggl(h,\frac{\partial h}{\partial\theta_1},\ldots,\frac{\partial h}{\partial\theta_{k-2}}\biggr) = \bigl(h_{(0)},h_{(1)},\ldots,h_{(k-2)}\bigr)
 \begin{pmatrix} 1 & 0 \\ 0 & D \end{pmatrix}, \quad\mbox{or}\ \ %
D=\biggl(h_{(i)}^\top\frac{\partial h}{\partial\theta_j}\biggr)_{1\le i,j\le k-2}.
\]
Now we have the ONB $h_{(i)}\otimes \psi_{(j)}$, $i=0,\ldots,k-2$, $j=0,\ldots,p-1$, of the ambient space $\R^n$ with $n=p(k-1)$.
Note that $\phi=h_{(0)}\otimes\psi_{(0)}$.

The tangent space $T_\phi M$ is spanned by
\begin{equation*}
 \frac{\partial\phi}{\partial x} = h\otimes\frac{\partial\psi}{\partial x}, \quad
 \frac{\partial\phi}{\partial\theta_i} = \frac{\partial h}{\partial\theta_i} \otimes\psi, \ \ i=1,\ldots,k-2.
\end{equation*}
The metric matrix of $T_\phi M$ with respect to the parameter $x,\theta_1,\ldots,\theta_{k-2}$ is
\begin{equation}
\label{metric}
 \begin{pmatrix} g & 0 \\ 0 & G \end{pmatrix}, \quad\mbox{where}\ \ %
G=\biggl(\biggl(\frac{\partial h}{\partial\theta_i}\biggr)^\top\biggl(\frac{\partial h}{\partial\theta_j}\biggr)\biggr)_{1\le i,j\le k-2} = D^\top D.
\end{equation}
$T_\phi M$ has the ONB
$h_{(0)}\otimes\psi_{(1)}$, $h_{(i)}\otimes\psi_{(0)}$, $i=1,\ldots,k-2$.

The normal space perpendicular to $T_\phi(\co(M))=T_\phi M\oplus\spn\{\phi\}$ is
\begin{align}
 N_{\phi}(\co(M))
= \spn\bigl\{ & h_{(0)}\otimes\psi_{(j)},\,j=2,\ldots,p-1;\ \nonumber \\
& h_{(i)}\otimes\psi_{(j)},\,i=1,\ldots,k-2,\,j=1,\ldots,p-1 \bigr\}.
\label{onb-normal}
\end{align}
 
The second order derivatives of $\phi=\phi(x,\theta)$ are
\[
 \frac{\partial^2\phi}{\partial x^2} = h\otimes\frac{\partial^2\psi}{\partial x^2}, \quad
\frac{\partial^2\phi}{\partial x\partial\theta_i} = \frac{\partial h}{\partial\theta_i}\otimes\frac{\partial\psi}{\partial x}, \quad
\frac{\partial^2\phi}{\partial\theta_i\partial\theta_j} = \frac{\partial^2 h}{\partial\theta_i\partial\theta_j}\otimes\psi.
\]
Taking the inner product of the second derivatives and the ONB of $N_{\phi}(\co(M))$ listed in (\ref{onb-normal}),
we see that the nonzero elements of the second fundamental form are
\[
 -\biggl(h\otimes\frac{\partial^2\psi}{\partial x^2}\biggr)^\top (h_{(0)}\otimes \psi_{(2)}) = -\biggl(\frac{\partial^2\psi}{\partial x^2}\biggr)^\top \psi_{(2)} = -\zeta,
\]
where
\[
 \zeta = \zeta(x) = \sqrt{\eta-\frac{\gamma^2}{g}-g^2}
\]
and
\[
 -\biggl(\frac{\partial h}{\partial\theta_i}\otimes\frac{\partial\psi}{\partial x}\biggr)^\top(h_{(j)}\otimes\psi_{(1)}) = -\biggl(\frac{\partial h}{\partial\theta_i}\biggr)^\top h_{(j)}\sqrt{g}=-D_{ji}\sqrt{g}.
\]

We renumber the ONB of $N_{\phi}(\co(M))$ as
\[
 N_1 = h_{(0)}\otimes\psi_{(2)},\ \ %
 N_i = h_{(i-1)}\otimes\psi_{(1)},\,i=2,\ldots,k-1,
\]
and $N_k,\ldots,N_{pk-p-k}$ are the other vectors.
Write $N(t) = \sum_{i=1}^{pk-p-k} N_i t_i$,
where $t=(t_1,\ldots,t_{pk-p-k})$.
Then,
\begin{align*}
 -\biggl(\frac{\partial^2\phi}{\partial x^2}\biggr)^\top N(t) =& -\zeta t_1, \\
 -\biggl(\frac{\partial^2\phi}{\partial x\partial\theta_i}\biggr)^\top N(t) =& -\sum_{j=1}^{k-2}D_{ji}t_{j+1}\sqrt{g}, \\
 -\bigg(\frac{\partial^2\phi}{\partial\theta_i\partial\theta_j}\biggr)^\top N(t) =& 0.
\end{align*}
Therefore, the second fundamental form (unnormalized version) in the direction $N(t)$ is
\begin{equation}
\label{2nd}
\left(\begin{array}{ccc}
 -\zeta t_1 & -(t_2,\ldots,t_{k-1}) D \sqrt{g} \\
 -D^\top \begin{pmatrix}t_{2}\\ \vdots \\ t_{k-1}\end{pmatrix} \sqrt{g} & 0
\end{array}\right).
\end{equation}
Multiplication of the inverse of the metric (\ref{metric}) enables us to obtain the normalized version of the second fundamental form.
Noting that
\[
 \begin{pmatrix} g & 0 \\ 0 & G \end{pmatrix}^{-1} =
 \begin{pmatrix} g & 0 \\ 0 & D^\top D \end{pmatrix}^{-1} =
 \begin{pmatrix} 1/\sqrt{g} & 0 \\ 0 & D^{-1} \end{pmatrix}
 \begin{pmatrix} 1/\sqrt{g} & 0 \\ 0 & (D^\top)^{-1} \end{pmatrix},
\]
we multiply
\[
 \begin{pmatrix} 1/\sqrt{g} & 0 \\ 0 & (D^\top)^{-1} \end{pmatrix}
 \quad\mbox{and}\quad
 \begin{pmatrix} 1/\sqrt{g} & 0 \\ 0 & D^{-1} \end{pmatrix}
\]
from the left and right to (\ref{2nd}), respectively, to obtain
\begin{align*}
\begin{pmatrix} 1/\sqrt{g} & 0 \\ 0 & (D^\top)^{-1} \end{pmatrix}
&
 \begin{pmatrix}
 -\zeta t_1 & -(t_2,\ldots,t_{k-1}) D \sqrt{g} \\
 -D^\top \begin{pmatrix} t_2 \\ \vdots \\ t_{k-1} \end{pmatrix} \sqrt{g} & 0
 \end{pmatrix}
\begin{pmatrix} 1/\sqrt{g} & 0 \\ 0 & D^{-1} \end{pmatrix} \\
& =\begin{pmatrix}
 -(\zeta/g) t_1 & -(t_2,\ldots,t_{k-1}) \\
 -\begin{pmatrix} t_2 \\ \vdots \\ t_{k-1} \end{pmatrix} & 0
\end{pmatrix}
= H(x,\theta;N(t)).
\end{align*}
This is the second fundamental form with respect to the orthonormal coordinates.
Now we have
\begin{equation}
\label{tr_e}
 \tr_e H(x,\theta;N(t)) =
\begin{cases}
  1                      & (e=0), \\
 -(\zeta/g) t_1          & (e=1), \\
 -\sum_{j=2}^{k-1} t_j^2 & (e=2), \\
 0                       & (\mbox{otherwise}).
\end{cases}
\end{equation}

Next, we evaluate the integral
\begin{align}
\label{integral}
\int_{v\in N_{\phi}(\co(M))\cap\S^{n-1}}
 \tr_e H(x,\theta;v) \,\dd v,
\end{align}
where $n=p(k-1)$, $\dd v$ is the volume element of $N_{\phi}(\co(M))\cap\S^{n-1}$, by following Section 4.2.2 of Kuriki and Takemura \cite{Kuriki-Takemura01}.
Recall that $d=\dim M=k-1$.

Because $N_{\phi}(\co(M))$ is a linear space of dimension $n-d-1 = p(k-1)-(k-1)-1 = pk-p-k$,
$N_{\phi}(\co(M))\cap\S^{n-1}$ is nothing but a ($pk-p-k-1$)-dimensional unit sphere.
Hence,
\[
 \int_{v\in N_{\phi}(\co(M))\cap\S^{n-1}} \,\dd v = \Vol(\S^{pk-p-k-1}) = \Omega_{pk-p-k}.
\]
Therefore, if $V$ is distributed as the uniform distribution on $N_{\phi}M\cap\S^{n-1}$, denoted by $\Unif(N_{\phi}M\cap\S^{n-1})$, 
 then $\mbox{(\ref{integral})} = \Omega_{pk-p-k}\times \E[\tr_e H(x,\theta;V)]$.

Suppose that
$T = (T_1,\ldots,T_{pk-p-k}) \sim \N_{pk-p-k}(0,I)$,
and let
$N(T) = \sum_{i=1}^{pk-p-k} N_i T_i$.
Then,
\[
 \Vert N(T)\Vert^2 = \sum_{i=1}^{pk-p-k} T_i^2 \sim \chi_{pk-p-k}^2
\quad\mbox{and}\quad
 V = \frac{N(T)}{\Vert N(T)\Vert} \sim \Unif(N_{\phi} M\cap\S^{n-1})
\]
are independently distributed.
Hence,
\[
\E[\tr_e H(x,\theta; N(T))]
= \E[\Vert N(T)\Vert^e \tr_e H(x,\theta,V)]
= \E[\Vert N(T)\Vert^e\} \E\{\tr_e H(x,\theta,V)],
\]
and
\[
 \E[\tr_e H(x,\theta,V)] = \frac{\E[\tr_e H(x,\theta; N(T))]}{\E\bigl[(\chi_{pk-p-k}^2)^{e/2}\bigr]}.
\]
From (\ref{tr_e}),
\[
 \E[\tr_e H(x,\theta; N(T))] =
 \begin{cases}
 1 & (e=0), \\
 0 & (e=1), \\
 -\sum_{i=2}^{k-1} \E[T_i^2] = -(k-2) & (e=2), \\
 0 & (\mbox{otherwise}),
 \end{cases}
\]
hence,
\[
 \E[\tr_e H(x,\theta;V)] =
 \begin{cases}
 1 & (e=0), \\
\displaystyle
 -\frac{k-2}{pk-p-k} & (e=2), \\
 0 & (\mbox{otherwise}).
 \end{cases}
\]
Therefore,
\[
 \int_{v\in N_{\phi}(\co(M))\cap\S^{n-1}}
 \tr_e H(x,\theta,v) \,\dd v = \begin{cases}
 \Omega_{pk-p-k} & (e=0), \\
\displaystyle
 -\frac{k-2}{pk-p-k} \Omega_{pk-p-k} & (e=2), \\
 0 & (\mbox{otherwise}). \end{cases}
\]

Note that the results are independent of $x$ and $\theta$.
This implies that the integral in (\ref{w}) with respect to the volume element
\[
 \dd u = \sqrt{g}\,\dd x\times |G(\theta)|^{1/2} \,\dd\theta_{1}\cdots \dd\theta_{k-2}
\] 
is simply multiplying the constant
\[
 \Vol(M) = |\Gamma|\times\Vol(\S^{k-2}) = |\Gamma|\Omega_{k-1}.
\]
Finally, from (\ref{w}) of Proposition \ref{prop:tube},
\begin{align*}
 w_{d+1} &= w_k =
 \frac{1}{\Omega_k\Omega_{pk-p-k}}\times |\Gamma| \Omega_{k-1}\Omega_{pk-p-k},
\\
 w_{d-1} &= w_{k-2} =
 \frac{1}{\Omega_{k-2}\Omega_{pk-p-k+2}}\times |\Gamma| \Omega_{k-1}\times\biggl(-\frac{k-2}{pk-p-k}\biggr)\Omega_{pk-p-k},
\end{align*}
and the other $w$'s are zero.
Simple calculations give
\begin{equation}
\label{w2} 
 w_k = -w_{k-2}
 = \frac{\Gamma (\frac{k}{2})}{\sqrt{\pi}\Gamma(\frac{k-1}{2})}|\Gamma|.
\end{equation}

\subsubsection*{Contribution of the boundary $\partial M$}

Here, we obtain the coefficients $w'_{d-e}$ in (\ref{w'}) when $M$ is given in (\ref{M}).

Suppose that $\mathcal{X}=[a,b]$.
Then,
\[
 \partial M = \{\phi(a,\theta) \mid \theta \in \Theta\} \sqcup \{\phi(b,\theta) \mid \theta \in \Theta\}.
\]
Let $x=a$ and $\theta\in\Theta$ be fixed.
Then, $\phi(a,\theta)\in\partial M$.
The metric of the boundary $\partial M$ at $\phi(a,\theta)$ is
\[
\biggl(\frac{\partial\phi}{\partial\theta_i}\biggr)^\top \biggl(\frac{\partial\phi}{\partial\theta_j}\biggr)\bigg|_{(a,\theta)} = (G(\theta))_{ij}.
\]

Note that
$\partial (\co(M)) = \co(\partial M)$. 
The support cone of $\co(M)$ at $\phi(a,\theta)\in\partial (\co(M))$ is
\[
 S_{\phi(a,\theta)}(\co(M)) = L \oplus K_1,
\]
where
\[
L = \spn\biggl\{h\otimes\psi,\ \frac{\partial h}{\partial\theta_i}\otimes\psi,\,i=1,\ldots,k-2\biggr\}, \quad
K_1 = \biggl\{\lambda\biggl({h\otimes\frac{\partial\psi}{\partial x}}\biggr) \mid \lambda \ge 0\biggr\}.
\]
This is a direct sum (the Minkowski sum) of a linear subspace and a cone.
To obtain its dual cone, the following lemma is useful.
\begin{lemma}
Let $K_{1}$ be a cone, and $L$ be a linear subspace.
Let $K = K_1 \oplus L$.
Then, the dual cone of $K$ is
$K^* = K_1^*\cap L^\perp$.
\end{lemma}
Because
\[
 K_1^* =
 \bigl\{ \lambda(h_{(0)}\otimes\psi_{(1)}) \mid \lambda\le 0 \bigr\} \oplus \spn\bigl\{ h_{(0)}\otimes\psi_{(1)} \bigr\}^\perp
\]
and
\[
 L^\perp = \spn\bigl\{
 h_{(i)}\otimes\psi_{(j)},\,i=0,\ldots,k-2,\,j=1,\ldots,p-1 \bigr\},
\]
we have
\begin{align*}
N_{\phi(a,\theta)}(\co(M))
= K^* = K_1^*\cap L^\perp
= \bigl\{ &\lambda(h_{(0)}\otimes \psi_{(1)}) \mid \lambda\le 0 \bigr\} \oplus  \spn\bigl\{ \\
& h_{(0)}\otimes \psi_{(j)},\,j=2,\ldots,p-1;\, \\
& h_{(i)}\otimes\psi_{(j)},\,i=1,\ldots,k-2,\,j=1,\ldots,p-1 \bigr\},
\end{align*}
with $\dim N_{\phi(a,\theta)}(\co(M)) = pk-p-k+1$.

The second fundamental form of $\co(\partial M)$ at $\phi(a,\theta)$ is
\begin{align}
\label{2nd-boundary}
 -\biggl(\frac{\partial^{2}\phi(a,\theta)}{\partial\theta_i\partial\theta_j}\biggr)^\top v =
 -\biggl(\frac{\partial^2 h}{\partial\theta_i\partial\theta_j} \otimes \psi\biggr)^\top v,\ \ v\in N_{\phi(a,\theta)}(\co(M)).
\end{align}
We can easily see that the second fundamental form (\ref{2nd-boundary}) is always zero.
Therefore, the contribution of the boundary to $w'_{d-e}$ in (\ref{w'}) is only to case of $e=0$.
That is, all $w'_i$ except $w'_{k-1}$ are zero.
The contribution of the boundary $\{\phi(a,\theta)\mid\theta\in\Theta\}$ to $w'_{k-1}$ is
\begin{align*}
 \frac{1}{\Omega_{k-1-0}\Omega_{p(k-1)-(k-1)+0}}
 \int_{\partial M}|G|^{1/2} \,\dd\theta
 \int_{N_{\phi(a,\theta)}(\co(M))\cap\S^{n-1}} \,\dd v
=& \frac{\Vol(\S^{k-2})\Vol(\mbox{half of }\S^{pk-p-k})}{\Omega_{k-1}\Omega_{pk-p-k+1}} \\
=& \frac{1}{2}.
\end{align*}
The contribution of the other boundary $\{ \phi(b,\theta) \mid \theta\in\Theta \}$ to $w_{k-1}$ has the same value of 1/2.
Moreover, if the number of connected components of $\Gamma$ exceeds one,
we need to select all boundaries.
Since the number of boundaries is $2\chi(\Gamma)$, we have
\begin{equation}
\label{w'2}
 w'_{k-1} = \frac{1}{2}\times 2\chi(\Gamma) = \chi(\Gamma).
\end{equation}

Substituting (\ref{w2}) and (\ref{w'2}) into (\ref{p-hat}) yields (\ref{p-hat2}).

\subsection{Proof of Theorem \ref{thm:critical_radius}}

We apply the formula (\ref{critical}) for $\theta_{\mathrm{c}}$ to the case where $M$ is given in (\ref{M}).

Let
\begin{align*}
u = \phi(\tilde x,\tilde\theta)
  = h(\tilde\theta) \otimes \psi(\tilde x), \quad
v = \phi(x,\theta) = h(\theta) \otimes \psi(x),
\end{align*}
and write
$h=h(\theta)$, $\tilde h=h(\tilde\theta)$,
$\psi=\psi(x)$, $\tilde\psi=\psi(\tilde x)$.
We discuss the two cases (i) $\psi\in\Int\Gamma$ and (ii) $\psi\in\partial\Gamma$ separately.

Case (i).
Suppose that $\psi(x)\in\Int\Gamma$.
Write $\phi_{\theta_i}=\{\partial h(\theta)/\partial\theta_i\}\otimes\psi(x)$,
$\phi_x=h(\theta)\otimes\psi_x(x)$,
$\psi_x=\partial\psi(x)/\partial x$.

The orthogonal projection matrix onto the space $T_{\phi}(\co(M))=\spn\{ \phi,\phi_{\theta_i},\phi_x \}$ is
\begin{align*}
P_v
& = \begin{pmatrix} \phi & \phi_{\theta_i} & \phi_x \end{pmatrix}_{p(k-1)\times k}
\left( \begin{pmatrix} \phi^\top \\ \phi_{\theta_i}^\top \\ \phi_x^\top \end{pmatrix}
 \begin{pmatrix} \phi & \phi_{\theta_i} & \phi_x \end{pmatrix} \right)^{-1}_{k\times k}
\begin{pmatrix} \phi^\top \\ \phi_{\theta_i}^\top \\ \phi_x^\top \end{pmatrix}_{k\times p(k-1)} \\
& = \begin{pmatrix} \phi & \phi_{\theta_i} & \phi_x \end{pmatrix}
 \begin{pmatrix} 1 & 0 & 0 \\ 0 & G(\theta) & 0 \\ 0 & 0 & g(x) \end{pmatrix}^{-1}
 \begin{pmatrix} \phi^\top \\ \phi_{\theta_i}^\top \\ \phi_x^\top \end{pmatrix} \\
& = \phi\phi^\top + (\phi_{\theta_i})G(\theta)^{-1}(\phi_{\theta_i})^\top
 + \frac{1}{g(x)}\phi_x\phi_x^\top \\
& =  h h^\top \otimes\psi\psi^\top + (I_{k-1} - h h^\top) \otimes \psi\psi^\top
 + \frac{1}{g} h h^\top \otimes \psi_x\psi_x^\top \\
& = I_{k-1} \otimes \psi\psi^\top + \frac{1}{g} h h^\top \otimes \psi_x\psi_x^\top.
\end{align*}
As $P_v^\perp(w)=I(w)-P_v(w)$,
\begin{equation}
\label{infarg}
\frac{(1-u^\top v)^2}{\Vert(I_n-P_v) u\Vert^2}
= \frac{\{1-(\tilde h\otimes\tilde\psi)^\top (h\otimes\psi)\}^2}
 {\Vert(I_{k-1}\otimes I_p -I_{k-1}\otimes\psi\psi^\top -\frac{1}{g} h h^\top\otimes\psi_x\psi_x^\top)(\tilde h\otimes\tilde\psi)\Vert^2}.
\end{equation}
Let
$s = \psi^\top(x)\psi(\tilde x) = \psi^\top\tilde\psi$, 
$r = \psi_x^\top(x)\psi(\tilde x) = \psi_x^\top\tilde\psi$,
and
$\alpha = h(\theta)^\top h(\tilde\theta) = h^\top \tilde h$.
In (\ref{infarg}), the numerator is $(1-\alpha s)^2$, and the denominator is 
\begin{align*}
\left\Vert \tilde h\otimes\tilde\psi-s \tilde h\otimes\psi-\frac{\alpha r}{g} h\otimes\psi_x\right\Vert^2
=& 1 + s^2 + \frac{\alpha^2 r^2}{g} - 2 s^2 - 2\frac{\alpha r}{g}\alpha r \\
=& 1 - s^2 - \frac{\alpha^2 r^2}{g} = 1 - s^2 - \alpha^2 t^2,
\end{align*}
where
$t = r/\sqrt{g} = \psi_x(x)^\top\psi(\tilde x)/\Vert \psi_x(x)\Vert$.
Hence,
\[
 \mbox{(\ref{infarg})} = \frac{(1-\alpha s)^{2}}{1-s^2-\alpha^2 t^2}.
\]

The infimum is taken over ``$x\ne \tilde x$ or $\theta\ne\tilde\theta$'', or equivalently, ``$x\ne \tilde x$ or $\alpha\ne 1$''.
However, when $x=\tilde x$ and $\alpha\ne 1$,
the argument of the infimum is $(1-\alpha)^2/0=\infty$.
Therefore, we can exclude case $x=\tilde x$ from the infimum argument.

Case (ii).
Suppose that $\psi(x)\in\partial\Gamma$.
Fix a point on the boundary
\[
 v = \phi(x,\theta) = h(\theta)\otimes\psi(x) \in \partial M.
\]
The support cone of $\co(M)$ at $v$ is
\[
 S_v(\co(M)) =
 \spn\{ \phi,\phi_{\theta_i} \}\oplus \{\lambda\varepsilon\phi_x \mid \lambda\ge 0 \}
\]
where $\varepsilon=\varepsilon(x)=1$ if $\psi_x$ is inward to $\Gamma$, $\varepsilon=-1$ if $\psi_x$ is outward to $\Gamma$.

The orthogonal projection operator onto the cone $S_v(\co(M))$ is
$w\mapsto P_v(w)$, where
\begin{align*}
P_v(w)
=& \phi\phi^\top w + \phi_{\theta}G^{-1}(\theta)\phi_{\theta}^{-1}w
 + \frac{\phi_x}{\Vert\phi_x\Vert^2}\max\{ 0,\varepsilon\phi_x^\top w \} \\
=& (I_{k-1} \otimes \psi\psi^\top)w
 + \frac{\phi_x}{\Vert\phi_x\Vert^2}\max\{0,\varepsilon\phi_x^\top w\}.
\end{align*}
Hence,
\begin{align*}
P_v^\perp(w)
= w-P_v(w)
&= w-(I_k \otimes \psi\psi^\top)w-\frac{h \otimes \psi_x}{g}\max\{0,\varepsilon(h \otimes \psi_x)^\top w\}.
\end{align*}
Substituting
$u = \phi(\tilde x,\tilde\theta) = \tilde h\otimes\tilde\psi$,
\[
 P_v^\perp(u-v) = \tilde h\otimes\tilde\psi - s \tilde h\otimes\psi
 - \frac{\max\{0,\varepsilon\alpha r\}}{g} h \otimes \psi_x,
\]
\begin{align*}
 \Vert P_v^\perp(u-v)\Vert^2
=& 1 + s^2 + \frac{\max\{0,\varepsilon\alpha r\}^2}{g} - 2 s^2 - 2\alpha r \frac{\max\{0,\varepsilon\alpha r\}}{g} \\
=& 1 - s^2 - \frac{\max\{0,\varepsilon\alpha r\}^2}{g} = 1 - s^2 - \max\{0,\varepsilon\alpha t\}^2,
\end{align*}
and we have
\begin{align*}
\frac{(1-u^\top v)^2}{\Vert P_v^\perp(u-v)\Vert^2}
= \frac{(1-\alpha s)^2}{1-s^2-\max\{0,\varepsilon \alpha t\}^2}.
\end{align*}
For the same reason as in case (i),
the infimum is taken over the set $x\ne \tilde x$ and $\alpha\in[-1,1]$.

\subsection{Proof of Theorem \ref{thm:local_critical_radius}}

We use the same notations as in the proof of Theorems \ref{thm:main} and \ref{thm:critical_radius}.
The local critical radius $\theta_{\mathrm{c,loc}}$ defined by (\ref{critical-local}) is rewritten as
\begin{align*}
\tan^2\theta_{\mathrm{c,loc}}
= \liminf_{|x-\tilde x|\to 0,\,\alpha\to 1}\frac{(1-\alpha s)^2}{1-s^2-\alpha^2 t^2}.
\end{align*}
Let $\tilde x=x+\Delta$ and $\alpha=1-\delta$, and consider $\Delta\to 0$ and $\delta\to 0$.
Write
$\psi_x=\partial\psi/\partial x$, $\psi_{xx}=\partial^2\psi/\partial x^2$,
etc., and
$g=\psi_x^\top\psi_x$, $\gamma=\psi_{xx}^\top\psi_x$, and $\eta=\psi_{xx}^\top\psi_{xx}$
as before.
Noting that
\begin{align*}
& 0=\dd (\psi^\top\psi)/\dd x=2\psi_x^\top\psi, \\
& 0=\dd^2(\psi^\top\psi)/\dd x^2=2\psi_{xx}^\top\psi+2\psi_x^\top\psi_x, \\
& 0=\dd^3(\psi^\top\psi)/\dd x^3=2\psi_{xxx}^\top\psi+6\psi_{xx}^\top\psi_x, \\
& 0=\dd^4(\psi^\top\psi)/\dd x^4=2\psi_{xxxx}^\top\psi+8\psi_{xxx}^\top\psi_x+6\psi_{xx}^\top\psi_{xx},
\end{align*}
we have
\[
 \psi_{xx}^\top\psi=-g, \quad \psi_{xxx}^\top\psi=-3\gamma, \quad \psi_{xxxx}^\top\psi+4\psi_{xxx}^\top\psi_x=-3\eta.
\]
Substituting these, we have
\begin{align*}
s = \psi(x)^\top\psi(\tilde x)
=&  \psi^\top \biggl( \psi + \psi_x\Delta + \frac{1}{2}\psi_{xx}\Delta^2 + \frac{1}{6}\psi_{xxx}\Delta^3 + \frac{1}{24}\psi_{xxxx}\Delta^4 \biggr) +o(\Delta^4) \\
=& 1 - \frac{1}{2}g\Delta^2 - \frac{1}{2}\gamma\Delta^3 + \frac{1}{24}\psi_{xxxx}^\top\psi\Delta^4 +o(\Delta^4),
\\
r = \psi_x(x)^\top\psi(\tilde x)
=& \psi_x^\top \biggl( \psi + \psi_x\Delta + \frac{1}{2}\psi_{xx}\Delta^2 + \frac{1}{6}\psi_{xxx}\Delta^3 \biggr) +o(\Delta^3) \\
=& g\Delta + \frac{1}{2}\gamma\Delta^2 + \frac{1}{6}\psi_{xxx}^\top\psi_x\Delta^3 +o(\Delta^3),
\\
t^2 = \frac{r^2}{g}
=& g\Delta^2 + \gamma\Delta^3 + \frac{1}{4g}\gamma^2\Delta^4 + \frac{1}{3}\psi_{xxx}^\top\psi_x\Delta^4 +o(\Delta^4),
\end{align*}
and
\begin{align*}
1-s^2-t^2
=& (1-s)\{2-(1-s)\}-t^2 = 2(1-s)-(1-s)^2-t^2 \\
=& 2 \biggl( \frac{1}{2}g\Delta^2 + \frac{1}{2}\gamma\Delta^3 - \frac{1}{24}\psi_{xxxx}^\top\psi\Delta^4 \biggr) - \biggl( \frac{1}{2}g\Delta^2 \biggr)^2 \\
& - \biggl( g\Delta^2 + \gamma\Delta^3 + \frac{1}{4g}\gamma^2\Delta^4 + \frac{1}{3}\psi_{xxx}^\top\psi_x\Delta^4 \biggr) +o(\Delta^4) \\
=& \frac{1}{4}\biggl( \eta-g^2-\frac{\gamma^2}{g} \biggr) \Delta^4 +o(\Delta^4)
= \frac{1}{4}\kappa g^2\Delta^4 +o(\Delta^4),
\end{align*}
where
\[
 \kappa=\kappa(x)=\frac{\eta}{g^2}-\frac{\gamma^2}{g^3}-1.
\]
Note that $\kappa$ is nonnegative because
\[
 0\le \det\begin{pmatrix}\psi^\top \\ \psi_x^\top \\ \psi_{xx}^\top \end{pmatrix}
          \begin{pmatrix}\psi & \psi_x & \psi_{xx} \end{pmatrix}
    = \det\begin{pmatrix} 1 & 0      & -g \\
                          0 & g      & \gamma \\
                         -g & \gamma & \eta \end{pmatrix} = \kappa g^3.
\]

For the order of $\delta$, we consider two cases:
(i) $\delta/\Delta^2 \sim g c$ ($0\le c<\infty$) and (ii) $\Delta^2/\delta \sim 0$.

For case (i), noting that $\alpha=1-\delta$, $1-s\sim g\Delta^2/2$, $t^2\sim g\Delta^2$,
\begin{align*}
& (1-\alpha s)^2 = \{1-s + \delta - \delta (1-s)\}^2 \sim (1-s + \delta)^2
 \sim \biggl( \frac{1}{2}g\Delta^2+g c\Delta^2 \biggr)^2 = \frac{1}{4}g^2(1+2c)^2\Delta^4,
\\
& 1 -s^2-\alpha^2 t^2 = 1 -s^2 -t^2 +2\delta t^2 -\delta^2t^2
\sim \frac{1}{4}g^2(\kappa + 8 c)\Delta^4,
\end{align*}
hence
\begin{align}
 \frac{(1-\alpha s)^2}{1-s^2-\alpha^2 t^2}
\sim \frac{(1+2c)^2}{\kappa+8c}.
\label{c}
\end{align}
We consider the minimum value of (\ref{c}) for $c\ge 0$.
For $c>-\kappa/8$, (\ref{c}) has a unique minimum value $1-\kappa/4$ at $c=(2-\kappa)/4$.
Therefore,
\[
\min_{c\ge 0}\frac{(1+2c)^2}{\kappa+8c} =
 \begin{cases}
 \displaystyle
 1-\frac{\kappa}{4} & (\kappa\le 2), \\
 \displaystyle
 \frac{1}{\kappa}   & (\kappa\ge 2).
 \end{cases}
\]

For case (ii),
\[
 (1 -\alpha s)^2 = \{1-s + \delta - \delta (1-s)\}^2 \sim \delta^2,
\] 
\[
 1 -s^2 -\alpha^2 t^2 = 1 -s^2 -t^2 +2\delta t^2 -\delta^2 t^2
 \sim 2 g\Delta^2\delta,
\]
and
\[
 \frac{(1-\alpha s)^2}{1-s^2-\alpha^2 t^2}
 \sim \frac{\delta}{2g\Delta^2} \to \infty.
\]

In summary, we have
\[
\tan^{2}\theta_{\mathrm{c,loc}} = \min\biggl\{
\inf_{x:\kappa(x)\le 2} \biggl\{1-\frac{\kappa(x)}{4}\biggr\},
\inf_{x:\kappa(x)\ge 2} \frac{1}{\kappa(x)} \biggr\}.
\]

\subsection{Coverage probability under model misspecification}
\label{sec:bound}

First, note that the best parameter under the model $\beta_i^\top f(x)$ is given by
$\beta_i^*= \Sigma X^\top g_i$,
where
\[
 X = \begin{pmatrix} f(x_1)^\top \\ \vdots \\ f(x_n)^\top \end{pmatrix}, \quad \Sigma = (X^\top X)^{-1} = \left( \sum_{i=1}^n f(x_i) f(x_i)^\top \right)^{-1}, \quad g_i = \begin{pmatrix} g_i(x_1) \\ \vdots \\ g_i(x_n) \end{pmatrix}.
\]
Let
\[
 y_i = \begin{pmatrix} y_{i1} \\ \vdots \\ y_{in} \end{pmatrix}, \quad \varepsilon_i = \begin{pmatrix} \varepsilon_{i1} \\ \vdots \\ \varepsilon_{i1} \end{pmatrix}.
\]
The least square estimator for $\beta_i^*$ is $\widehat\beta_i^* = \Sigma X^\top y_i$, which is distributed as $\N_p(\beta_i^*,\Sigma)$.
Let $b_{1-\alpha}$ be the threshold for $1-\alpha$ bands when the assumed model is the true model.
The approximate value of $b_{1-\alpha}$ can be obtained by the tube method.

On the other hand, when the true model is $g_i(x)$, the coverage probability becomes
\begin{align}
& \Pr\Biggl( \biggl|\sum_{i=1}^k c_i (\widehat\beta_i^*)^\top f(x) - \sum_{i=1}^k c_i g_i(x)\biggr| \le b_{1-\alpha} \sqrt{f(x)^\top\Sigma f(x)} \ \ \mbox{for all }x\in\mathcal{X},\ c\in\mathcal{C}\Biggr) \nonumber \\
&= \Pr\Biggl( \max_{x\in\mathcal{X},\,c\in\mathcal{C}}\frac{\sum_{i=1}^k c_i (\widehat\beta_i^*)^\top f(x) - \sum_{i=1}^k c_i g_i(x)}{\sqrt{f(x)^\top\Sigma f(x)}} \le b_{1-\alpha} \Biggr) \nonumber \\
&=
\Pr\Biggl( \max_{x\in\mathcal{X},\,c\in\mathcal{C}} \biggl[
\frac{\sum_{i=1}^k c_i (\widehat\beta_i^*-\beta_i^*)^\top f(x)}{\sqrt{f(x)^\top\Sigma f(x)}}+\frac{\sum_{i=1}^k c_i \{(\beta_i^*)^\top f(x) - g_i(x)\}}{\sqrt{f(x)^\top\Sigma f(x)}}\biggr] \le b_{1-\alpha} \Biggr).
\label{pr}
\end{align}

Noting that, for the two functions $h_1(y)$ and $h_2(y)$ on $\mathcal{Y}$,
if $\max_{y\in\mathcal{Y}} (-h_i(y)) = \max_{y\in\mathcal{Y}} h_i(y)$,
then
\[
 \max_y h_1(y) \le \max_y (h_1(y)+h_2(y)) + \max_y (-h_2(y)) = \max_y (h_1(y)+h_2(y)) + \max_y h_2(y),
\]
hence,
\[
 \max_y h_1(y) - \max_y h_2(y) \le \max_y (h_1(y)+h_2(y)) \le \max_y h_1(y) + \max_y h_2(y).
\]
Therefore, (\ref{pr}) is bounded below and above by
$1-\P(b_{1-\alpha}-\delta)$ and $1-\P(b_{1-\alpha}+\delta)$, respectively, where
\begin{align*}
 \delta
=& \max_{x\in\mathcal{X},\,c\in\mathcal{C}}\frac{\sum_{i=1}^k c_i \{(\beta_i^*)^\top f(x) - g_i(x)\}}{\sqrt{f(x)^\top\Sigma f(x)}} \nonumber \\
&= \max_{x\in\mathcal{X}}\sqrt{\frac{\sum_{i=1}^k [(\beta_i^*)^\top f(x) - g_i(x)
 -\frac{1}{k}\sum_{i=1}^k \{(\beta_i^*)^\top f(x) - g_i(x)\}]^2}{f(x)^\top\Sigma f(x)}}
\end{align*}
as given in (\ref{delta}), and
\[
 \P(b) = \Pr\Biggl( \max_{x\in\mathcal{X},\,c\in\mathcal{C}} \frac{\sum_{i=1}^k c_i (\widehat\beta_i^*-\beta_i^*)^\top f(x)}{\sqrt{f(x)^\top\Sigma f(x)}} \ge b \Biggr).
\]
An upper bound for the bias of coverage probability for a $1-\alpha$ confidence band is
\[
 \max\bigl\{ \alpha-\P(b_{1-\alpha}+\delta), \P(b_{1-\alpha}-\delta)-\alpha \bigr\},
\]
which is approximated by
\begin{align*}
\Delta =
 \max\bigl\{ \alpha-\PP\bigl(b_{\mathrm{tube},1-\alpha}+\delta\bigr), \PP\bigl(b_{\mathrm{tube},1-\alpha}-\delta\bigr)-\alpha \bigr\}
\end{align*}
in (\ref{Delta}), where $\PP(b)$ is the tube approximation formula for $P(b)$ given in (\ref{p-hat2}),
and $b_{\mathrm{tube},1-\alpha}$ is the solution of $\PP(b)=\alpha$.

Note that (\ref{pr}) is
\[
 \Pr\Biggl( \max_{x\in\mathcal{X}} \frac{\sum_{i=1}^k [(\widehat\beta_i^*)^\top f(x) - g_i(x)
 -\frac{1}{k}\sum_{i=1}^k \{(\widehat\beta_i^*)^\top f(x) - g_i(x)\}]^2}{f(x)^\top\Sigma f(x)} \le b_{1-\alpha}^2 \Biggr),
\]
which is used for the simulation study.

\subsubsection*{Acknowledgment}

The authors thank Professor Toshihiko Shiroishi and Dr.\ Toyoyuki Takada of the National Institute of Genetics, Japan, for providing the dataset of Takada et al.\ \cite{Takada-etal08}.

\renewcommand{\refname}

\end{document}